\documentclass{amsart}

\usepackage{amsthm,amssymb,latexsym,amsmath}
\usepackage[all]{xypic}
\usepackage{dsfont}

\usepackage[symbol]{footmisc}

\usepackage{ytableau}

\newcommand{\bb}{\mathbf{B}}

\newcommand{\catA}[1]{{\mathfrak A}}
\newcommand{\catI}[1]{{\mathfrak I}}
\newcommand{\catS}[1]{{\mathfrak S}}
\newcommand{\cpx}{\mathbb{C}}

\newcommand{\im}{{\rm Im}~}

\newcommand{\oh}{{\mathcal O}}

\newcommand{\p}[1]{{\mathbb{P}^{#1}}}

\DeclareMathOperator{\B}{B}
\DeclareMathOperator{\codim}{{codim}}

\DeclareMathOperator{\End}{End}
\DeclareMathOperator{\ext}{Ext}
\DeclareMathOperator{\Hom}{Hom}

\DeclareMathOperator{\ho}{H}

\DeclareMathOperator{\coker}{coker}

\newtheorem{theorem}{Theorem}[section]

\newtheorem{proposition}[theorem]{Proposition}
\newtheorem{lemma}[theorem]{Lemma}
\newtheorem{corollary}[theorem]{Corollary}

\newtheorem{remark}[theorem]{Remark}
\newtheorem{example}[theorem]{Example}
\newtheorem{definition}[theorem]{{\bf Definition}}
\newtheorem*{conjecture}{{\bf Conjecture}}

\begin{document}

\title{On the fixed locus of framed instanton sheaves on $\mathbb{P}^{3}$ }

\author{Abdelmoubine Amar Henni}

\date{}

\maketitle

\vspace{1cm}

\begin{abstract}
Let $\mathbb{T}$ be the three dimensional torus acting on $\mathbb{P}^{3}$ and $\mathcal{M}^{\mathbb{T}}_{\mathbb{P}^{3}}(c)$ be the fixed locus of the corresponding action on the moduli space of rank $2$ framed instanton sheaves on $\mathbb{P}^{3}.$ In this work, we prove that $\mathcal{M}^{\mathbb{T}}_{\mathbb{P}^{3}}(c)$ consist only of non locally-free instanton sheaves whose double dual is the trivial bundle $\mathcal{O}_{\p3}^{\oplus 2}$. Moreover, we relate these instantons to Pandharipande-Thomas stable pairs and give a classification of their support. This allows to compute a lower bound on the number of components of $\mathcal{M}^{\mathbb{T}}_{\mathbb{P}^{3}}(c).$

\end{abstract}

\vspace{1cm}

\tableofcontents

\vspace{1cm}

\section{Introduction}
\label{IN}

\emph{Framed instanton sheaves} have been the subject of study for more than four decades and by many authors of different backgrounds. One of the main reasons, is that they reflect a deep connection between algebraic geometry and mathematical physics; in the late 70's, Atiyah, Drinfeld, Hitchin and Manin fully classified the Yang-Mills anti-self-dual solutions, known as \emph{instantons} \cite{BPST,DK,ADHM}, on the four sphere $S^{4}.$ The classification was given, first, by relating instantons with certain holomorphic bundles on the projective space $\mathbb{P}^{3},$ over $\mathbb{C},$ by means of Penrose-Ward correspondence. Then by using Horrocks \emph{monads} \cite{Horr}, introduced in the late $60$'s, the authors got linear algebraic data, called the \emph{ADHM} data. Donaldson, then, discovered that \emph{framed } instantons on the four sphere $S^4$ correspond to some \emph{framed} holomorphic bundles on the projective plane $\mathbb{P}^{2}$ \cite{Don}. Moreover, during the $90$'s Nakajima considered framed sheaves in order to provide a compactification \cite{Nak1,Nak}, of the moduli space of framed instanton bundles on surfaces. This led to the computation of many invariants \cite{Nak, Nak2}, such as Betti numbers and Euler characteristic of these moduli spaces, on one hand, and a connection to representation theory by means of Quiver varieties \cite{Nak3} and the infinite Heisenberg Algebra \cite{Nak, Bara}, on the other hand. It is worth to mention that the rank $1$ case gives an explicit description of the Hilbert scheme of points on $\mathbb{C}^{2}$ in terms of ADHM data, and is a basic model for the computations in the higher rank cases \cite{Bara, Bruzzo-Pog}

On $\mathbb{P}^{3},$ the particular rank $2$ instanton bundles corresponds to the $SU(2)$ gauge theoretic instantons on the four sphere $S^{4}.$ Their moduli space have been studied for decades and some of its properties remained illusive for a long time. For instance, its irreducibility has been proved just few years ago, by Tikhomirov \cite{Tikho1,Tikho2}. Also, not long ago, its smoothness was showed by Jardim and Verbitsky \cite{JV}. Recently, there have been some interest in its compactification by using torsion-free sheaves. \cite{JMT1,JMT2,JG}. 

In this work, we are interested in the moduli space of rank $2$ framed instanton sheaves $\mathcal{M}_{\mathbb{P}^{3}}(c)$, on the three dimensional projective space $\mathbb{P}^{3}.$ More precisely we study its fixed locus $\mathcal{M}^{\mathbb{T}}_{\mathbb{P}^{3}}(c)$ with respect to the torus action inherited by the natural one on $\mathbb{P}^{3}.$ We shall see that every fixed torsion-free instanton sheaf $E$ is an extension (non trivial in general) of ideal sheaves $\mathcal{I}_{\mathcal{C}}$ and $\mathcal{I}_{\mathcal{Z}},$ where $\mathcal{I}_{\mathcal{C}}$ is the ideal sheaf of a non-reduced Cohen Macaulay curve $\mathcal{C},$ whose underlying reduced support is the line $l_{0}=Z(z_{2}=z_{3}=0),$ i. e., the unique line that is fixed by the action of $\mathbb{T}$ and does not intersect the framing line $l_{\infty}$ at infinity, and $\mathcal{I}_{\mathcal{Z}}$ is the ideal sheaf of points supported on $p_{0}=[1;0;0;0]$ or $/$ and $p_{1}=[0;1;0;0],$ in $l_{0}\subset\p3.$ Moreover, using the fact that the double dual of such $E$ is the trivial bundle $\mathcal{O}^{2}_{\mathbb{P}^{3}},$ we also show that every corresponding quotient $\mathcal{Q}:=\mathcal{O}^{2}_{\mathbb{P}^{3}}/E$ is a pure sheaf of dimension $1$ on the curve $\mathcal{C}.$ These quotient sheaves $\mathcal{Q}$ are special cases of {\em rank $0$ instanton sheaves} \cite[\S6.1]{HL}. A similar phenomenon, that occurs on $\p2,$ is the fact that the fixed points in $\mathcal{M}_{\p2}(r,c),$ under the toric action inherited from the one on $\p2,$ split as the sum of ideal sheaves of points, all with the same topological support given by the origin $\lbrack 0;0;1\rbrack$ \cite[\S3]{Nak2, Bruzzo-Pog}. The difference is that the set of fixed points, in the $\p3$ case, might not be isolated in general, in other words, there might be continuous families of them.

\vspace{0.5cm}

This paper is organized as follows; in Section \ref{ADHM}, we recall the notion of \emph{ADHM data} and their stabilities on $\mathbb{P}^{3}$ and how it relates to framed instanton sheaves through Horrocks monads. In section \ref{T-action} we briefly describe the inherited action, of the three dimensional torus $\mathbb{T},$ on the ADHM data. In particular, we show that, for non vanishing second Chern class, the fixed framed instantons are strictly torsion-free sheaves, that their double dual is trivial and that their singularity locus is pure, of dimension $1.$ A different proof can be found in \cite{JG} for non fixed instantons.

\vspace{0.3cm}

In Section \ref{PT-pairs}, we move on to give a relation of these fixed instanton sheaves with Pandharipande-Thomas stable pairs \cite{Panda-Thom,Panda-Thom2}. More precisely, we show that to every fixed framed rank $2$ instanton sheaf $\mathcal{F}$, on $\p3,$ one may associate a PT-stable pair $(\mathcal{Q},s).$ Furthermore, we show that the Euler characteristic $\chi(\mathcal{M}_{\mathbb{P}^{3}}(c))$ is zero, for any $c>0.$ 

\vspace{0.3cm}

Section \ref{Multiple structures} is devoted to completely classify the Cohen-Macaulay supports $\mathcal{C}$ associated to the fixed PT-stable pair $(\mathcal{Q},s),$ i.e, coming from a fixed instanton sheaf of rank $2,$ in $\p3.$ This is achieved by using results on monomial multiple structures provided by Vatne \cite{Vatne}. 

\vspace{0.3cm}

Finally, in Section \ref{sheaves-multi-struc}, we show that a lower bound for the number of irreducible components of $\mathcal{M}^{\mathbb{T}}_{\mathbb{P}^{3}}(c),$ is given by the number of partitions of $c.$ Moreover, we use results provided by Dr\'ezet \cite{Drezet1, Drezet2}, in order to give an explicit description of the first canonical filtration of the rank $0$ instanton sheaf $\mathcal{Q},$ for $c=2,$ and when the support is a primitive monomial double structure. When $c=3,$ we describe the $0$ instanton sheaf $\mathcal{Q},$ and its filtrations for the three possible monomial structures, in particular the first non primitive case is given in Theorem \ref{non-primitive}. Using these results, we show that $\mathcal{M}^{\mathbb{T}}_{\mathbb{P}^{3}}(3),$ has $7$ components. For $c=1,$ we also compute the dimension of the tangent space and the obstruction at the specific fixed stable pair $(\mathcal{Q},s).$

We wonder whether these fixed points can arise as degenerations of locally free framed instantons, i. e., if the fixed enumerated components intersect the closure of the framed locally free instanton moduli. We think that this problem is related to reachability of sheaves, on multiple structure \cite{Drezet3} and hope to address this problem in future work.

\bigskip\bigskip

\section{ADHM Data and instanton sheaves}\label{ADHM}

In this section we will gather useful results about ADHM data and instanton sheaves. Mostly, this material can be found in \cite{henni1,FJ2,J-i}. We consider in $\p3$ the homogeneous coordinates $[z_0 :z_1:z_2:z_3]\in\p3$ and the line $\ell_{\infty}$ given by the equations $z_0=z_1=0$. Set
$$
\ho_{\p1}=\langle z_0,z_1\rangle\subset\ho^0(\p3).
$$

Let $V$ and $W$ be complex vector spaces of dimension, respectively, $c$ and $r$. Set
$$
\bb:={\rm End}(V)^{\oplus 2}\oplus{\rm Hom}(W,V)\
$$

and consider the affine spaces

$$
\bb_{\p1}=\bb_{\p1}(W,V)=\bb_{\p1}(r,c):=\bb\otimes\ho_{\p1}
$$

A point of $\mathbf{B}_\p1$ will be called in this paper an \emph{ADHM datum over $\p1$}. 

One can write a point of $X\in\mathbf{B}_\p1$ as
$$
X=(A,B,I)
$$
where the above components are
$$
A = A_{0}\otimes z_0 + A_{1}\otimes z_1
$$
$$
B = B_{0}\otimes z_0 + B_{1}\otimes z_1
$$
$$
I = I_0\otimes z_0 + I_1\otimes z_1\ \ \ \ \ \ \ \
J = J_0\otimes z_0 + J_1\otimes z_1\
$$

\

\noindent with $A_{i}, B_{i}, \in {\rm End}(V)$, $I_i \in {\rm Hom}(W,V)$ and $J_i \in {\rm Hom}(V,W),$ $i=0,1$. Hence we naturally regard $A,B\in {\rm Hom}(V,V\otimes \ho_\p1)$, and also $I\in {\rm Hom}(W,V\otimes \ho_\p1)$ and $J\in {\rm Hom}(V,W\otimes \ho_\p1)$.

For any $P\in\p1$ we define the \emph{evaluation maps} given on generators by
\begin{gather*}
\begin{matrix}
{\rm ev}_{P}^1: &\bb_\p1 & \longrightarrow & \mathbb{P}(\bb)\\
               & X_i\otimes z_i              & \longmapsto     & [z_i(P)X_i]
\end{matrix}
\end{gather*}
Note that $z_i(P)\in\cpx$ depends on a choice of trivialization of $\oh_\p1(1)$ at $P$ but the class on projective space does not. We set $X_P:={\rm ev}_{P}^{1}(X).$ In particular, $A_P$, $B_P$, $I_P$ and $J_P$ are defined as well. For any subspace $S\subset V$, we are able to naturally well define the subspaces $A_P(S),B_P(S),I_P(W)$ and $\ker J_P$ of $V$. 

We also consider the following {\em stability} and {\em costability} conditions:

\begin{definition}\label{Def-stab1}\cite{henni1}
Let $X=(A,B,I)\in\mathbf{B}_{\mathbb{P}^{1}}$. Let also $P$ be a point in $\mathbb{P}^{3}$.
\begin{enumerate}
\item[(i)] $X$ is said {\rm globally}  {\rm weak stable} if there is no proper subspace $S\subset V$ of dimension $1$ for which hold the inclusions  $A_P(S),B_P(S),I_P(W)\subset S,$ for every $P\in \p1$;
\item[(ii)] $X$ is said {\rm globally}  {\rm weak costable} if  there is no nonzero subspace $S\subset V$  of dimension $1$ for which hold the inclusions $A_P(S),B_P(S)\subset S\subset\ker J_P,$ for every $P\in \p1$.
\end{enumerate}
\end{definition}

We define $\mathbf{B}_\p1^{\rm gws}$  as the subsets of $\mathbf{B}_\p1$ consisting of globally weak stable ADHM data over $\mathbb{P}^{3}.$ In a similar way, one defines $\mathbf{B}_\p1^{\rm gwc}.$ Clearly, both of them are open subsets of $\mathbf{B}_\p1,$ in the Zariski topology.

\bigskip

An \emph{instanton sheaf} on $\mathbb{P}^{n}$ is a torsion free coherent sheaf $E$ with $c_{1}(E)=0$ satisfying the following cohomological conditions:
\begin{itemize}
\item[(i)] for $n\geq2,$ $\ho^{0}(\mathbb{P}^{n},E(-1))=\ho^{n}(\mathbb{P}^{n},E(-n))=0;$
\item[(ii)] for $n\geq3,$ $\ho^{1}(\mathbb{P}^{n},E(-2))=\ho^{n-1}(\mathbb{P}^{n},E(1-n))=0;$
\item[(iii)] for $n\geq4,$ $\ho^{p}(\mathbb{P}^{n},E(-k))=0$ for $2\leq p\leq n-1,$ $\forall k.$
\end{itemize}
The second Chern class $c:=c_{2}$ is called the \emph{charge} of $E$, and one can check that $c=-\chi(E)=h^1(E(-1))$. An instanton sheaf is said to be of \emph{trivial splitting type} if there exists a line $\ell$ in $\mathbb{P}^{3}$ such that the restriction $E|_{\ell}$ of $E$ on $\ell$ is trivial. A particular choice of trivialisation $\phi:E|_{\ell}\to\mathcal{O}|^{\oplus r}_{\ell}$ is called a \emph{framing}, and the pair $(E,\phi)$ is called a \emph{framed} instanton sheaf.

Now, we consider framed instantons on $\mathbb{P}^{3},$ and we fix the line $\ell_{\infty},$ given in the beginning of this section, as the framing line. Moreover, from \cite[Proposition 3.6]{henni1}, \cite{FJ2,J-i}, we have that framed rank $r$ instantons $E$ of charge $c$ are cohomologies of monads of the form:
\begin{equation}\label{monad}\mathcal{M}:\quad\xymatrix{V\otimes\mathcal{O}_{\mathbb{P}^{3}}(-1)\ar[r]^{\alpha\quad\quad}& (V\oplus V\oplus W)\otimes\mathcal{O}_{\mathbb{P}^{3}}\ar[r]^{\quad\quad\beta}& V\otimes\mathcal{O}_{\mathbb{P}^{3}}(1) }\end{equation}
$V$ is a $c-$dimensional vector space and can be identified with some homology group of $E$ twisted by a some differential sheaf, via the Beilinson spectral sequence construction of the monad \cite[\S 3]{henni1}. $W$ is $r-$dimensional space (this can be identified with $\mathbb{C}^{r},$ given a fixed basis, due to the framing). The maps $\alpha$ and $\beta$ are given by:
\begin{displaymath} 
\alpha=\left(\begin{array}{l} A_{0}z_{0}+A_{1}z_{1}+\mathds{1}z_{2}\\B_{0}z_{0}+B_{1}z_{1}+\mathds{1}z_{3} \\ J_{0}z_{0}+J_{1}z_{1}\end{array}\right),
\end{displaymath}
\begin{displaymath}
\beta=\left( -B_{0}z_{0}-B_{1}z_{1}-\mathds{1}z_{3}, A_{0}z_{0}+A_{1}z_{1}+\mathds{1}z_{2}, I_{0}z_{0}+I_{1}z_{1}\right),
\end{displaymath}
where $A_{0},A_{1},B_{0},B_{1}\in\End(V),$ $I_{0},I_{1}\in\Hom(W,V)$ and $J_{0},J_{1}\in\Hom(V,W).$ These matrices satisfy the following equations 
\begin{equation}
\label{equcad}
\begin{array}{c}
[ A_{0} , B_{0} ] + I_0J_0 = 0 \\

[ A_{1} , B_{1} ] + I_1J_1 = 0  \\

[ A_{0} , B_{1} ] + [ B_{0} , A_{1} ] + I_1J_2 + I_2J_1 = 0 
\end{array}
\end{equation}
which are equivalent to the complex condition $\beta\circ\alpha=0,$ in the monad $\mathcal{M}.$ Moreover there is a group action of $G=Gl(V),$ on the above data, given by
\begin{equation}
\begin{array}{c}
A_{i}\to gA_{i}g^{-1} \\
B_{i}\to gB_{i}g^{-1} \\
I_{i}\to gI_{i} \\
J_{i}\to J_{i}g^{-1}
\end{array}
\end{equation}
for $g\in G$ and $i=0,1.$

\begin{color}{green}
\end{color}
An easy verification, by using of additivity of the Chern character, shows that $c_{3}(E)=0.$  

We denote by $\mathcal{V}_{\mathbb{P}^{3}}(c,r)$ the space of the ADHM data satisfying the equations \eqref{equcad} and in which one can define the following subvarieties, acording to stabilities in Definition \ref{Def-stab1}
$\mathcal{V}^{\rm gws}_{\mathbb{P}^{3}}(c,r) \subset\mathcal{V}^{\rm st}_{\mathbb{P}^{3}}(c,r).$ Here, we recall that subscript $gws$ stands for globally weak stable and subscript $st$ stands for stable, meaning that there is no subspace $S\subsetneq V$ such that $A(S),B(S), I(W)\subset S\otimes\ho_{\p1}.$ The proof of the inclusion is discussed in \cite[Section 2.1]{henni1}.

\bigskip 

For an ADHM datum $X=(A_{0},B_{0},I_{0},J_{0},A_{1},B_{1},I_{1},J_{1}),$ consider the following algebraic set $$D_{X}=\{z\in\mathbb{P}^{3}| \alpha_{X}\textnormal{ is not injective}\}.$$ Note that we always have $\codim(D_{X})\geq2,$ by the framing condition.
A simpler version of \cite[Proposition 3.3 $\&$ Proposition 3.4]{henni1} can be written as the
\begin{theorem}\label{sheaf-stability}
The complex \eqref{monad} is a monad if and only if the corresponding ADHM datum is globally weak stable, and in this case $E,$ the middle cohomology of the monad, is torsion-free. Moreover, $E$ is a locally free framed instanton sheaf if and only if the ADHM datum $X$ is globally weak costable.
\
\end{theorem}

\section{Torus action on the ADHM data}\label{T-action}

Now, we consider the following standard torus action of $\mathbb{T}:=\mathbb{T}^{3}$ on $\mathbb{P}^{3}$ given by 
$$F_{t}:\xymatrix@R-2pc{\mathbb{T}\times\mathbb{P}^{3}&\ar[r]&\mathbb{P}^{3} \\ ((t_{1},t_{2},t_{3}),z) &\ar@{|->}[r]& [z_{0};t_{1}z_{1};t_{2}z_{2};t_{3}z_{3}] }$$

This action can be lifted to the space of ADHM data as the following: let ${\bf T}:=\mathbb{T}\times\tilde{T},$ where $\tilde{T}$ is the maximal torus of $GL(W),$ given by elements of the form $e=diag(e_{1},\cdots,e_{r}).$ Let $\gamma_{e_{1},\cdots,e_{r}}$ be the isomorphism $\mathcal{O}|_{\ell}^{r}\ni(w_{1},\cdots,w_{r})\to(e_{1}w_{1},\cdots,e_{r}w_{r})\in\mathcal{O}|_{\ell}^{r}.$ For a framed instanton sheaf $(E,\phi:E|_{\ell}\to\mathcal{O}|_{\ell}^{r})$ one can define the following  $(t,e_{1},\cdots,e_{r})\cdot(E,\phi)=((F_{t}^{-1})^{\ast}E,\phi'),$ where $\phi'$ is given by the composition $$(F_{t}^{-1})^{\ast}E|_{\ell}\overset{(F_{t}^{-1})^{\ast}\phi}{\to}(F_{t}^{-1})^{\ast}\mathcal{O}|_{\ell}^{2}\longrightarrow\mathcal{O}|_{\ell}^{r}\overset{\gamma_{e_{1},\cdots,e_{r}}}{\to}\mathcal{O}|_{\ell}^{r}.$$

\begin{proposition}
The above action can be identified with the action on the ADHM data given by:

\begin{align}\label{torus-action}
&A_{0}\to t_{2}A_{0}  \qquad   &A_{1}\to  t_{1}^{-1}t_{2}A_{1} \\
&B_{0}\to  t_{3}B_{0}  \qquad    &B_{1}\to  t_{1}^{-1}t_{3}B_{1} \nonumber\\
&J_{0}\to t_{3}eJ_{0}   \qquad   &J_{1}\to t_{1}^{-1}t_{3}eJ_{1} \nonumber\\
&I_{0}\to t_{2}I_{0}e^{-1}  \qquad   &I_{1}\to t_{1}^{-1}t_{2}I_{1}e^{-1} \nonumber
\end{align}
Moreover, the ADHM equations \eqref{equcad} and stability conditions are preserved.
\end{proposition}

\begin{proof}
Since any framed instanton sheaf $E$ is the middle cohomology of a monad as in \eqref{monad}, then the pull back $(F_{t}^{-1})^{\ast}E$ is the cohomology of a similar monad with maps $\alpha$ and $\beta$ given as below:
\begin{displaymath} 
\alpha=\left(\begin{array}{l} A_{0}z_{0}+A_{1}t_{1}^{-1}z_{1}+\mathds{1}t_{2}^{-1}z_{2}\\B_{0}z_{0}+B_{1}t_{1}^{-1}z_{1}+\mathds{1}t_{3}^{-1}z_{3} \\ J_{0}z_{0}+J_{1}t_{1}^{-1}z_{1}\end{array}\right),
\end{displaymath}
\begin{displaymath}
\beta=\left( -B_{0}z_{0}-B_{1}t_{1}^{-1}z_{1}-\mathds{1}t_{3}^{-1}z_{3}, A_{0}z_{0}+A_{1}t_{1}^{-1}z_{1}+\mathds{1}t_{2}^{-1}z_{2}, I_{0}z_{0}+I_{1}t_{1}^{-1}z_{1}\right),
\end{displaymath}

Under the isomorphism $$\begin{array}{l}
 V\otimes\mathcal{O}_{\mathbb{P}^{3}}\\ \oplus \\ V\otimes\mathcal{O}_{\mathbb{P}^{3}} \\ \oplus \\
 W\otimes\mathcal{O}_{\mathbb{P}^{3}}
\end{array}\ni \left(\begin{array}{l}
 v_{1}\\v_{2} \\ w
\end{array}\right)\to\left(\begin{array}{l}
 t_{3}^{-1}v_{1}\\t_{2}^{-1}v_{2} \\t_{2}^{-1} w
\end{array}\right)$$
the kernel of $\beta$ is sent to the kernel of
{\small
\begin{displaymath}\left( -(t_{3}B_{0})z_{0}-(t_{1}^{-1}t_{3}B_{1})z_{1}-\mathds{1}z_{3}, (t_{2}A_{0})z_{0}+(t_{1}^{-1}t_{2}A_{1})z_{1}+\mathds{1}z_{2}, (t_{2}I_{0})z_{0}+(t_{1}^{-1}t_{2}I_{1})z_{1}\right)
\end{displaymath}
}
and the image of $\alpha$ is sent to the image of 
\begin{displaymath}
\frac{1}{t_{2}t_{3}}\left(\begin{array}{l} (t_{2}A_{0})z_{0}+(t_{1}^{-1}t_{2}A_{1})z_{1}+\mathds{1}z_{2}\\(t_{3}B_{0})z_{0}+(t_{1}^{-1}t_{3}B_{1})z_{1}+\mathds{1}z_{3} \\ (t_{3}J_{0})z_{0}+(t_{1}^{-1}t_{3}J_{1})z_{1}\end{array}\right)
\end{displaymath}
Composing with the action of $\gamma_{e_{1},e_{2}}$ on the framing, the assertion follows.

\end{proof}

Now we consider the moduli space $\mathcal{M}_{\p3}(r,c):=\mathcal{V}^{st}_{\p3}(r,c)/G$ (this quotient makes sense by means of \cite[Section 2.3]{henni1}); a datum $[X]$ is invariant under the toric action if and only if there exists an element $g_{t}\in G$ such that $t\cdot X=g_{t}\cdot X.$ In other words, $[X]=[A_{0},B_{0},I_{0},J_{0},A_{1},B_{1},I_{1},J_{1}]$ is $\mathbb{T}-$invariant if and only if there exists a map $$\begin{array}{lll}\theta:\mathbb{T}& \to & G \\ \hspace{0,5cm} t & \mapsto & \theta(t)=g_{t}\end{array}$$ such that

\begin{align}\label{fixed}
&t_{2}A_{0}= g_{t}A_{0}g^{-1}_{t} \qquad   && t_{1}^{-1}t_{2}A_{1}= g_{t}A_{1}g^{-1}_{t}\\
&t_{3}B_{0}= g_{t}B_{0}g^{-1}_{t} \qquad    && t_{1}^{-1}t_{3}B_{1}=g_{t}B_{1}g^{-1}_{t} \nonumber\\
&t_{3}J_{0} = J_{0}g^{-1}_{t} \qquad   && t_{1}^{-1}t_{3}J_{1}= J_{1}g^{-1}_{t}\nonumber\\
&t_{2}I_{0}= g_{t}I_{0} \qquad   && t_{1}^{-1}t_{2}I_{1}=g_{t}I_{1}\nonumber
\end{align}

\bigskip

\begin{lemma}
If $[X]$ is fixed by the torus $\mathbb{T},$ then we have $J_{0}=J_{1}=0.$ Moreover, $X$ is not globally weak costable.
\end{lemma}
\begin{proof}
Suppose $[X]$ is fixed by the torus $\mathbb{T},$ and let $t=(t_{1},t_{2},t_{3}).$ Then one has $J_{0}I_{0}=(J_{0}g^{-1}_{t})(g_{t}I_{0})=(t_{3}J_{0})(t_{2}I_{0})=t_{2}t_{3}J_{0}I_{0},$ hence $J_{0}I_{0}=0.$ In the same way, one shows that $J_{\alpha}I_{\beta}=0,$ for all $\alpha,\beta=0,1.$
Moreover, for $A=z_{0}A_{0}+z_{1}A_{1}$ $B=z_{0}B_{0}+z_{1}B_{1},$ $I=z_{0}I_{0}+z_{1}I_{1}$ and $J=z_{0}J_{0}+z_{1}J_{1}$ such that $\lbrack A,B\rbrack+IJ=0,$ $\forall z_{0},z_{1},$ one has

$$JBA=J\lbrack A,B\rbrack+JAB=J(-IJ)+JAB=-\underbrace{(JI)}_{0}J+JAB=JAB.$$ 
Thus, by induction, it follows that for any product $\hat{C}=C_{\alpha_{1}}\cdot C_{\alpha_{2}}\cdots BA\cdots C_{\alpha_{m}},$ where $\alpha_{i}=0,1,$ $\forall i=1,\cdots,m$ and $$C_{\alpha_{i}}=\left\{\begin{array}{ll}A & \alpha_{i}=0\\ B & \alpha_{i}=1, \end{array}\right.,$$ one has $J\hat{C}=C_{\alpha_{1}}\cdot C_{\alpha_{2}}\cdots BA\cdots C_{\alpha_{m}}.$ Hence, for any such product, we have:
\begin{equation}\label{commute}
    J\hat{C}=JA^{l}B^{m},
\end{equation}
where $l$ and $m$ are the numbers os $A'$s and $B'$s, respectively, appearing in $\hat{C}.$ On the other hand, we have 
\begin{align} 
J_{0}A_{0}^{l}B_{0}^{m}I_{0}&=J_{0}g^{-1}_{t}g_{t}A_{0}^{l}g^{-1}_{t}g_{t}B_{0}^{m}g^{-1}_{t}g_{t}I_{0}\nonumber \\
&=(t_{3}J_{0})(t_{2}^{l}A_{0}^{l})(t^{m}_{3}B_{0}^{m})(t_{2}I_{0}) \nonumber \\
&=(t_{2}^{m+1}t_{3}^{l+1})J_{0}A_{0}^{l}B_{0}^{m}I_{0}, 
\end{align} 
for all $t\in\mathbb{T}.$ Hence $J_{0}A_{0}^{l}B_{0}^{m}I_{0}=0.$ In the same way, it follows that 
\begin{equation}\label{sandwich0} J_{\alpha_{1}}A_{\alpha_{2}}^{l}B_{\alpha_{3}}^{l}I_{\alpha_{4}}=0,
\end{equation}
for all $\alpha_{i}=0,1.$
By the stability condition, we have that $V\otimes\ho_\p1$ is generated by the action of $C_{\alpha_{1}}^{l}C_{\alpha_{2}}^{l}$ on $I(w_{1})$ and $I(w_{2}),$ where $<w_{1},w_{2}>=\mathbb{C}^{2},$ Then every vector $v\in V\otimes\ho_\p1$ is of the form $\Sigma_{\alpha_{k}} C_{\alpha_{1}}\cdots C_{\alpha_{m}}I(w_{1})+\Sigma_{\alpha_{k}} C'_{\alpha_{1}}\cdots C'_{\alpha_{m}}I(w_{2})$
Hence 
\begin{align}
J_{\alpha}v&= \Sigma_{\alpha_{k}} J_{\alpha}C_{\alpha_{1}}\cdots C_{\alpha_{m}}I(w_{1})+\Sigma_{\alpha_{k}} J_{\alpha}C'_{\alpha_{1}}\cdots C'_{\alpha_{m}}I(w_{2})\nonumber \\
&=0 \textnormal{ by } \eqref{commute} \textnormal{ and } \eqref{sandwich0}. \nonumber
\end{align}

Therefore, both $J_{0}$ and $J_{1}$ vanish identically. Moreover, it follows that the datum $X$ is not globally weak costable.

\end{proof}

\newpage

\begin{theorem}\label{Tfixed-tfree}
A $\mathbb{T}-$fixed framed rank $r$ instanton of charge $c$ on $\mathbb{P}^{3}$ is

\begin{itemize}
\item[(i)] strictly torsion free if $c>0,$ or
\item[(ii)] equal to the trivial bundle if $c=0.$
\end{itemize}
\end{theorem}
\begin{proof}
By the correspondence in Theorem \ref{sheaf-stability}, we conclude that, for $c>0,$ the instanton $E$ corresponding to $T-$fixed datum $X$ is not locally free. From the framing, we conclude that the singularity set of the sheaf is at least $2-$codimensional, hence the instanton sheaf is torsion-free, in this case.  If $c=0,$ the only instanton sheaf is the trivial bundle, which is clearly fixed by the torus action.
\end{proof}

\bigskip

In the rank $2$ case we have the following result:

\begin{theorem}\label{strict-torsion-free}
Let $E$ be a rank $2$ torsion-free instanton sheaf on $\mathbb{P}^{3}.$ Then
\begin{itemize}
\item[(i)] The singularity set $Sing(E)$ of $E$ is purely $1-$dimensional.
\item[(ii)] Furthermore, if $E$ is $\mathbb{T}-$fixed, then
\begin{itemize}
\item[(a)] its double dual $E^{\ast\ast}$ is the trivial locally free instanton sheaf $\mathcal{O}^{\oplus 2}_{\p3},$ and 
\item[(b)] its singularity locus is topologically supported on the rational line given by $z_{2}=z_{3}=0.$ Moreover, the matrices $A_{i},B_{i},$ for $i=0,1,$ in the corresponding ADHM datum are nilpotent.
 
\end{itemize}
\end{itemize}
\end{theorem}
\begin{proof}
Suppose $E$ is reflexive, then $E$ should be locally free since it is of rank two and has third chern class $c_{3}(E)=0$ \cite[Proposition 2.6]{Hart}. This contradicts Theorem \ref{Tfixed-tfree} for $c\neq0.$ Hence the singularity set $Sing(E)$ of $E$ is $1-$dimensional. it remains to check purity. This is done by showing that the quotient sheaf $\mathcal{Q}:=E^{\ast\ast}/E$ is pure. The sheaf $\mathcal{Q}$ is suported in codimension $2,$ thus we have $\mathcal{E}xt^{q}(\mathcal{Q},\mathcal{O}_{\mathbb{P}^{3}}(-4))=0,$ for $q=0,1.$ Moreover, by \cite[Proposition 1.1.10]{Huy}, $\mathcal{Q}$ is pure if, and only if, $\codim (\mathcal{E}xt^{3}(\mathcal{Q},\mathcal{O}_{\mathbb{P}^{3}}(-4)))\geq 3+1=4.$ In other words, we need to show that $\mathcal{E}xt^{3}(\mathcal{Q},\mathcal{O}_{\mathbb{P}^{3}}(-4)))$ is the zero sheaf.

Note that $\mathcal{Q}$ is a $1-$dimensional sheaf, so by \cite[Proposition 1.1.6]{Huy} we have $\codim(\mathcal{E}xt^{3}(\mathcal{Q},\mathcal{O}_{\mathbb{P}^{3}}(-4)))\geq 3.$ Hence $\mathcal{E}xt^{3}(\mathcal{Q},\mathcal{O}_{\mathbb{P}^{3}}(-4))$ is, supposedly, supported on a zero-dimentional subscheme in $\mathbb{P}^{3},$ lying inside $Sing(E).$   

By Serre-Grothendieck duality can write  
\begin{align}
\ext^{3}(\mathcal{Q},\mathcal{O}_{\mathbb{P}^{3}})&=\ext^{0}(\mathcal{O}_{\mathbb{P}^{3}},\mathcal{Q}(-4))^{\ast} \nonumber \\
&=\ho^{0}(\mathbb{P}^{3},\mathcal{Q}(-4))^{\ast}.\nonumber
\end{align}

Now we will show that $\ho^{0}(\mathbb{P}^{3},\mathcal{Q}(-4))^{\ast}=0.$, This will be achieved by using the long exact sequence in cohomology, associated to the short exact sequence $$0\to E(-4)\to E^{\ast\ast}(-4)\to\mathcal{Q}(-4)\to0.$$ In fact, one has 

\begin{equation}\label{homo-sequence}
\to\ho^{0}(\mathbb{P}^{3},E^{\ast\ast}(-4))\to\ho^{0}(\mathbb{P}^{3},\mathcal{Q}(-4))\to\ho^{1}(\mathbb{P}^{3},E(-4)).
\end{equation}

\vspace{0.2cm}
{\bf Claim $1$: } $\ho^{1}(\mathbb{P}^{3},E(-4)).$
\vspace{0.2cm}

 From the monad, associated to $E,$ one has the sequences 

\begin{equation}\label{beta-sequence2}
0\to\ker\beta\to\mathcal{O}_{\mathbb{P}^{3}}^{\oplus (2c+2)}\stackrel{\beta}{\to} \mathcal{O}_{\mathbb{P}^{3}}(1)^{\oplus c}\to0,
\end{equation}

\begin{equation}\label{beta-sequence}
0\to\mathcal{O}_{\mathbb{P}^{3}}(-1)^{\oplus c}\to\ker\beta\to E\to0.
\end{equation}

By using \eqref{beta-sequence2} and its dual, one can verify that $\ho^{1}(\mathbb{P}^{3},\ker\beta\otimes \mathcal{O}_{\mathbb{P}^{3}}(-4))=0$ and $\ho^{0}(\mathbb{P}^{3},(\ker\beta)^{\ast}(-4))=0.$

Twisting the sequence \eqref{beta-sequence} by $\mathcal{O}_{\mathbb{P}^{3}}(-4),$ and analysing the long sequence in cohomology, it is not difficult to check that $\ho^{1}(\mathbb{P}^{3},E(-4))=0.$

\vspace{0.2cm}
{\bf Claim $2$: } $\ho^{0}(\mathbb{P}^{3},E^{\ast\ast}(-4)).$
\vspace{0.2cm}

On the other hand, by dualizing this sequence, one gets
$$0\to E^{\ast\ast}(-4)\to(\ker\beta)^{\ast}\otimes\mathcal{O}_{\mathbb{P}^{3}}(-4)\to\mathcal{O}_{\mathbb{P}^{3}}(-3)^{\oplus c}\to\mathcal{E}xt^{1}(E,\mathcal{O}_{\mathbb{P}^{3}}(-4))\to0.$$

Where we used the fact that $E^{\ast\ast}\cong E^{\ast},$ since $E^{\ast}$ is a rank $2$ reflexive sheaf on $\mathbb{P}^{3},$ with trivial first Chern class. Breaking this sequence into 
$$0\to E^{\ast\ast}(-4)\to(\ker\beta)^{\ast}\otimes\mathcal{O}_{\mathbb{P}^{3}}(-4)\to J\to0$$

and using the long exact sequence in cohomology one gets the desired result, since $\ho^{0}(\mathbb{P}^{3},(\ker\beta)^{\ast}(-4))$ is trivial.

It follows from the proved claims and the sequence \eqref{homo-sequence} that $\ext^{3}(\mathcal{Q},\mathcal{O}_{\mathbb{P}^{3}})=0,$ as desired.
On the other hand, from the local-to-global spectral sequence one has $$\ext^{3}(\mathcal{Q},\mathcal{O}_{\mathbb{P}^{3}})=\ho^{0}(\mathbb{P}^{3},\mathcal{E}xt^{3}(\mathcal{Q},\mathcal{O}_{\mathbb{P}^{3}}))\oplus\ho^{1}(\mathbb{P}^{3},\mathcal{E}xt^{2}(\mathcal{Q},\mathcal{O}_{\mathbb{P}^{3}})).$$

 both contributing terms $$\ho^{0}(\mathbb{P}^{3},\mathcal{E}xt^{3}(\mathcal{Q},\mathcal{O}_{\mathbb{P}^{3}})),\qquad\ho^{1}(\mathbb{P}^{3},\mathcal{E}xt^{2}(\mathcal{Q},\mathcal{O}_{\mathbb{P}^{3}})),$$ to the local-to-global spectral sequence, must be trivial. Finally, observe that $\dim\ho^{0}(\mathbb{P}^{3},\mathcal{E}xt^{3}(\mathcal{Q},\mathcal{O}_{\mathbb{P}^{3}}))=0$ is the length of the sheaf $\mathcal{E}xt^{3}(\mathcal{Q},\mathcal{O}_{\mathbb{P}^{3}}),$ which must be zero since any sheaf supported on a zero-dimensional subscheme of $\mathbb{P}^{3},$ with zero length is the zero sheaf. Hence $\mathcal{Q}$ is pure.

To see that $c_{3}(E^{\ast\ast})=0,$ it suffice to show the following 

\vspace{0.2cm}
{\bf Claim $3$: } $E^{\ast\ast}$ is the cohomology of a monad
\vspace{0.2cm}

Clearly, $E^{\ast\ast}$ satisfies $\ho^{0}(\mathbb{P}^{3},E(-1))=\ho^{3}(\mathbb{P}^{3},E(-3))=0,$ since it is also framed. Moreover, it also satisfies $\ho^{2}(\mathbb{P}^{3},E^{\ast\ast}(-2))=0,$ since this sits in the sequence
$$\ho^{2}(\mathbb{P}^{3},E(-2))\to\ho^{2}(\mathbb{P}^{3},E^{\ast\ast}(-2))\to\ho^{2}(\mathbb{P}^{3},\mathcal{Q}(-2)),$$
in which $\ho^{2}(\mathbb{P}^{3},E(-2))=0,$ by instanton definition (page $4$), and $\ho^{2}(\mathbb{P}^{3},\mathcal{Q}(-2))=0$ from the fact that $\dim Supp(\mathcal{Q})=1.$

The dual of the complex \eqref{monad}, is a perverse instanton sheaf of trivial splitting type \cite[\S 5.4]{henni1}, \cite{HL} whose zero'th cohomology is $E^{\ast},$ and first cohomology is $\mathcal{E}xt^{1}(E,\mathcal{O}_{\mathbb{P}^{3}})$ Then by \cite[Theorem 5.13]{henni1}, The hypercohomology $\mathbb{H}^{1}(\mathbb{P}^{3},\mathcal{M}^{\ast}(-2))$ vanishes, and since this is just $\ho^{1}(\mathbb{P}^{3},E(-2))\oplus\ho^{0}(\mathbb{P}^{3},\mathcal{E}xt^{1}(E^{\ast},\mathcal{O}_{\mathbb{P}^{3}})(-2)),$ one gets in particular that $\ho^{1}(\mathbb{P}^{3},E^{\ast}(-2))=0.$ Again, one has $E^{\ast\ast}\cong E^{\ast},$ hence the double dual satisfy the definition of instanton sheaf, and this proves the claim.

Finally, recall that the double dual of any sheaf is reflexive. Thus $E^{\ast\ast}$ is a reflexive framed instanton sheaf, which fixed by the torus action on $\mathbb{P}^{3}.$ Moreover, by \cite[Proposition 2.6]{Hart}, $E^{\ast\ast}$ should be locally free, since $c_{3}(E^{\ast\ast})=0.$ But according to Theorem \ref{Tfixed-tfree} we should have $E^{\ast\ast}=\mathcal{O}^{\oplus 2}_{\mathbb{P}^{3}},$ since the non trivial $\mathbb{T}-$fixed istantons should be strictly torsion-free.

To conclude the proof of item ${\rm (b)}$ we note that $E$ is torsion free, by Theorem \ref{Tfixed-tfree} and its singularity set is purely $1-$dimensional, by item ${\rm (i)}.$ Moreover, by the framing condition it does not intersect the framing line. In particular, the singularity set is also $\mathbb{T}-$invariant. But the only invariant codimension $2$ subscheme of $\mathbb{P}^{3},$ as a toric variety, which does not intersect the framing line is supported on the line $[z_{0},z_{1},0,0].$

\vspace{0.5cm}

The singularity set is the locus on which the map $\alpha,$ in the monad \eqref{monad}, is not injective. In particular, one can characterize it in terms of the eigenvalues equations $$ \det[(A_{0}z_{0}+A_{1}z_{1})+z_{2}\mathds{1}]=0, \hspace{1cm} \det[(B_{0}z_{0}+B_{1}z_{1})+z_{3}\mathds{1}]=0.$$

But we just showed that all the corresponding eigenvalues $z_{2}, z_{3}$ must be $0.$ Hence the matrices $(A_{0}z_{0}+A_{1}z_{1})$ and $(B_{0}z_{0}+B_{1}z_{1})$ must be nilpotent, for all $z_{0},z_{1},$ and consequently, the result follows.

\end{proof}

We note that a proof for item ${\rm (i)}$ can also be found in \cite{JG} ,however, we gave our own proof for completeness.Moreover, this is shorter version concerned mainly with $\mathbb{T}-$fixed locus.

From the above, we see that if $[X]\in\mathcal{M}_{\mathbb{P}^{3}(r,c)}$ is a $\mathbb{T}-$fixed point, then is represented by a datum $X=(A_{0},B_{0},I_{0},A_{1},B_{1},I_{1})$ satisfying the equations:
\begin{equation}
\label{equa-fixed}
\begin{array}{c}
[ A_{0} , B_{0} ] = 0 \\

[ A_{1} , B_{1} ] = 0  \\

[ A_{0} , B_{1} ] + [ B_{0} , A_{1} ] = 0 
\end{array}
\end{equation}

\section{Quotients and PT-stable pairs}\label{PT-pairs}

In this section we will adopt the following viewpoint:
Let $E$ be a $\mathbb{T}-$invariant torsion-free instanton sheaf of rank $2,$ then $E$ fits in the short exact sequence $$0\to E\to\mathcal{O}^{2}_{\mathbb{P}^{3}}\to\mathcal{Q}\to0.$$

The Hilbert polynomials of sheaves involved in this sequence are $$P_{E}(m)=\frac{1}{3}m^{3}+2m^{2}+(\frac{11}{3}-c)m+(2-2c),\quad P_{\mathcal{Q}}(m)=cm+2c.$$  Since every such  sheaf $E$ is given by a datum $X\in\mathcal{V}^{\mathbb{T}}_{\mathbb{P}^{3}}(c),$ one can think of $\mathcal{M}^{\mathbb{T}}_{\mathbb{P}^{3}}(c)$ as an open subset of the scheme ${\bf Quot}_{\mathcal{O}^{2}_{\mathbb{P}^{3}},\lbrack l_{0}\rbrack }^{\lbrack cm+2c\rbrack},$ which parametrizes quotients $\mathcal{O}^{2}_{\mathbb{P}^{3}}\twoheadrightarrow\mathcal{Q}$ with $\mathcal{Q}$ is $\mathbb{T}-$fixed $1-$dimensional pure sheaf, topologically supported on the fixed line $l_{0}: \mathbb{P}^{1}\hookrightarrow\mathbb{P}^{3},$ not intersecting the framing line $l_{\infty},$ that is $l_{0}$ is given by $\lbrack z_{0};z_{1}\rbrack\mapsto\lbrack z_{0};z_{1};0;0\rbrack.$ The instanton cohomological conditions on $E$ imply the $\mathcal{Q}$ should satisfy $\ho^{0}(\mathbb{P}^{3},\mathcal{Q}(-2))=\ho^{1}(\mathbb{P}^{3},\mathcal{Q}(-2))=0.$ Obviously these are open conditions in flat families by semicontinuity.

Recall that a {\em rank $0$ instanton sheaf} is a pure sheaf of codimension $2$ satisfying $\ho^{0}(\mathbb{P}^{3},\mathcal{Q}(-2))=\ho^{1}(\mathbb{P}^{3},\mathcal{Q}(-2))=0.$ This definition was introduced in \cite[\S6.1]{HL}.

\begin{lemma}\label{stab-J}\cite[Lemma 6]{JMT}
Every rank $0$ instanton sheaf is $\mu-$semi-stable.
\end{lemma}

\begin{proof}
Let ${\bf T}$ a subsheaf of $\mathcal{Q}$ with Hilbert polynomial $P_{{\bf T}}(m)=am+b.$ Note that $\ho^{0}(\mathbb{P}^{3},\mathcal{Q}(-2))=0$ implies that $\ho^{0}(\mathbb{P}^{3},{\bf T}(-2))=0.$ Thus $P_{{\bf T}}(-2)=-2a+b=-\ho^{1}(\mathbb{P}^{3},{\bf T}(-2))\leq0.$ Hence $\mu({\bf T})=\frac{b}{a}\leq2=\frac{2c}{c}=\mu(\mathcal{Q}).$  
\end{proof}
Thus the quotient $\mathcal{Q}$ is a rank $0$ instanton sheaf, and hence $\mu-$semi-stable.

\bigskip

Following \cite{Panda-Thom}, let $q\in\mathbb{Q}[x]$ a degree $1$ polynomial with positive leading coefficient. For $n\in\mathbb{Z}$ and $\beta\in\ho^{2}(\mathbb{P}^{3},\mathbb{Z}),$ we will consider pairs $\mathcal{O}_{\mathbb{P}^{3}}\xrightarrow[]{s}\mathcal{Q}$, on $\mathbb{P}^{3},$ where $\mathcal{Q}$ is a pure sheaf, of dimension $1$ on $\mathbb{P}^{3},$ with Hilbert polynomial $$\chi(\mathcal{Q}(m))=m\cdot\beta+n,$$ The polynomial $q$ is viewed as a stability parameter, and $s$ is a non-zero section. We also let $r_{\mathcal{T}}$ denote, for any sheaf $\mathcal{T},$ the leading coefficient of its Hilbert polynomial. 
Since $\mathcal{Q}$ is pure, then any proper subsheaf $\mathcal{T}$ of $\mathcal{Q}$ is also pure of the same dimension. Therefore $r_{\mathcal{T}}>0.$ We say that the pair $(\mathcal{Q},s)$ is {\em $q-$(semi-)stable} if, for any proper subsheaf $\mathcal{T}\subset\mathcal{Q},$ the inequality 
\begin{equation}\label{q-stab1}
\frac{\chi(\mathcal{T}(m))}{r_{\mathcal{T}}}< (\leq) \frac{\chi(\mathcal{Q}(m))+q(m)}{r_{\mathcal{Q}}}, \qquad m>>0,
\end{equation}
holds, and for any proper subsheaf $\mathcal{T}\subset\mathcal{Q},$ through which the section $s$ factors, the inequality 
\begin{equation}\label{q-stab2}
\frac{\chi(\mathcal{T}(m))+q(m)}{r_{\mathcal{T}}}< (\leq) \frac{\chi(\mathcal{Q}(m))+q(m)}{r_{\mathcal{Q}}}, \qquad m>>0,
\end{equation}
holds. The moduli space of such $q-$(semi)stable pairs is denoted by $P_{n}^{q}(\mathbb{P}^{3},\beta),$ and was constructed by Le Potier, in \cite{LePo}, using GIT. Moreover the pair $(\mathcal{Q},s)$ is said to be {\em stable} if is stable in the large $q$ limit, i.e., for sufficiently large coefficients of $q.$ In this case, we will drop the superscript $q$ and denote the moduli of such stable pairs, simply, by $P_{n}(\mathbb{P}^{3},\beta).$ For more details about its construction and the fact that it  has a perfect obstruction, and hence a well defined virtual class  one might consult \cite{Panda-Thom} and references therein.

We also recall the following 
\begin{lemma}\label{pandha-Thom}\cite[Lemma 1.3]{Panda-Thom} \\
A pair $(\mathcal{Q},s)$ is stable (in the large $q$ limit) if, and only if,
\begin{itemize}
\item[(i)] $\mathcal{Q}$ has pure dimension $1.$ 
\item[(ii)] The cokernel of $s$ has dimension zero.
\end{itemize}
\end{lemma}

\vspace{0.5cm}

In what follows, we will say that $(\mathcal{Q},s)$ is a {\em stable $0$ instanton pair} if it is stable (in the large $q$ limit) and $\mathcal{Q}$ is a rank $0$ instanton sheaf. Moreover, as in \cite[p. 402]{JMT}, an instanton sheaf $E$ will be called {\em quasitrivial} if its double dual is the trivial sheaf. 

Recall also that if $\mathcal{Q}$ is associated to a quasitrivial torsion-free instanton, then one has the following commutative diagram:
\begin{equation}\label{s-map}
\xymatrix@R-1.3pc@C-1.3pc{ & &0 \ar[d]& &  \\
 & &\mathcal{O}_{\mathbb{P}^{3}}\ar[d]\ar[dr]^{s} & &  \\
 0\ar[r]& E\ar[r]& \mathcal{O}_{\mathbb{P}^{3}}^{2}\ar[r]\ar[d]& \mathcal{Q}\ar[r]& 0 \\
 & &\mathcal{O}_{\mathbb{P}^{3}}\ar[d] & &  \\
 & & 0& &  \\
}
\end{equation}

This define a section $s$ of $\mathcal{Q}.$ Let us put $\mathcal{G}:=\im(s)$ and $\mathcal{I}:=\ker(s).$ Then $\mathcal{I}$ is an ideal sheaf in $\mathcal{O}_{\mathbb{P}^{3}}$ of a subscheme $S$ of pure dimension $1$ in $\mathbb{P}^{3},$ with structure sheaf $\mathcal{O}_{\mathcal{C}}=\mathcal{G}.$ Moreover, if $E$ is $\mathbb{T}-$fixed, then $\mathcal{Q}$ is a $\mathbb{T}-$fixed rank $0$ instanton sheaf. It follows from Theorem \ref{strict-torsion-free} that the theoretical support of $S$ is exactly the line $l_{0},$ defined by the locus $(z_{2}=z_{3}=0),$ in $\mathbb{P}^{3}.$

\begin{proposition}\label{inst-pair-stab}
Let $(\mathcal{Q},s)$ be $0$ instanton pair associated to a rank $2$ quasitrivial instanton sheaf $E$ on $\mathbb{P}^{3}.$ Then
\begin{itemize}
\item[(i)] $(\mathcal{Q},s)$ is stable;
\item[(ii)]  Moreover, if $E$ is framed and $\mathbb{T}-$fixed, then $E\in\ext^{1}(\mathcal{I}_{\mathcal{Z}},\mathcal{I}_{\mathcal{C}}),$ where $\mathcal{I}_{\mathcal{C}}$ is a $\mathbb{T}-$fixed ideal sheaf of a subscheme of  pure dimension $1$ and $\mathcal{I}_{\mathcal{Z}}$ is $\mathbb{T}-$fixed ideal sheaf of a zero-dimensional subscheme and neither $\mathcal{Z},$ nor $\mathcal{C},$ intersects the line $l_{\infty}.$
\end{itemize} 
\end{proposition}

\begin{proof}

\begin{itemize}
\item[(i)] By \cite[Corollary 5]{JMT}, it follows that $\coker s$ is zero dimensional and since $\mathcal{Q}$ is pure, the result follow by Lemma \ref{pandha-Thom}

\vspace{0.3cm}

\item[(ii)] First, notice that since $\mathcal{Q}$ is a rank $0$ instanton, one can complete the commutative diagram \eqref{s-map} above to get the following one:
\begin{equation}\label{ideal-ext}
\xymatrix@R-1.3pc@C-1.3pc{ &0\ar[d] &0 \ar[d]&0\ar[d] &  \\
 0\ar[r]&\mathcal{I}_{\mathcal{C}}\ar[r]\ar[d] &\mathcal{O}_{\mathbb{P}^{3}}\ar[r]\ar[d] &\mathcal{O}_{\mathcal{C}}\ar[r]\ar[d] &0  \\
 0\ar[r]& E\ar[d]\ar[r]& \mathcal{O}_{\mathbb{P}^{3}}^{2}\ar[r]\ar[d]& \mathcal{Q}\ar[r]\ar[d]& 0 \\
0\ar[r] &\mathcal{I}_{\mathcal{Z}}\ar[r]\ar[d] &\mathcal{O}_{\mathbb{P}^{3}}\ar[d]\ar[r] &\mathcal{Z}\ar[r]\ar[d] & 0 \\
 & 0& 0&0 &  \\
}
\end{equation}
where  $\mathcal{O}_{\mathcal{C}}:=\im(s),$ $\mathcal{I}_{\mathcal{C}}:=\ker(s)$ and $\mathcal{Z}:=\coker(s).$
By \cite[Corollary 5]{JMT} it follows that, the rank $2$ instanton sheaf $E$ belongs to $\ext^{1}(\mathcal{I}_{\mathcal{Z}},\mathcal{I}_{\mathcal{C}}),$ where $\mathcal{I}_{\mathcal{C}}$ is an ideal sheaf of a subscheme in of  pure dimension $1$ and $\mathcal{I}_{\mathcal{Z}}$ is ideal sheaf of a zero-dimensional subscheme in $\mathbb{P}^{3}.$ Furthermore, if $E$ is $\mathbb{T}-$fixed, then so are $\mathcal{I}_{\mathcal{Z}}$ and $\mathcal{I}_{\mathcal{C}},$ since $\mathcal{Z}$ and $\mathcal{C}$ are fixed. Finally, $\mathcal{C}$ and $Supp(\mathcal{Z})$ are supported on $Sing(E),$ and by item ${\rm (b)},$ of Theorem \ref{strict-torsion-free}, $Sing(E)$ has vacuous intersection with the framing line $l_{\infty},$ so the result follows.
\end{itemize}
\end{proof}

\vspace{0.5cm}

We end this section with the following 

\begin{theorem}\label{Euler-char}
The Euler characteristic of $\mathcal{M}_{\mathbb{P}^{3}}(c)$ is $\chi(\mathcal{M}_{\mathbb{P}^{3}}(c))=0, \quad \forall c>0.$ 

Moreover if $c=1,$ then the Poincar\'e Polynomial of $\mathcal{M}_{\mathbb{P}^{3}}(1)$ is $$\mathcal{P}_{\mathcal{M}_{\mathbb{P}^{3}}(1)}(t)=\sum_{i=0}^{13}(1-\delta_{1,i}-\delta_{12,i})t^{i}$$
\end{theorem}

Before giving a proof, we recall now some useful definitions mostly from \cite{Serre};

\begin{definition}
Let $Y$ be an algebraic space endowed with a right (or left) action of a group $G,$ and let $\pi:Y\longrightarrow X$ be a morphism from $Y$ to the algebraic space $X.$ We call the triple $(G,Y,X)$ (or just $Y$) a fibered system if $\pi$ satisfies $\pi(x\cdot g)=\pi(x)$ for all $x\in X$ and $g\in G.$  
\end{definition}

A fibered system $Y$ is called \emph{locally trivial} (\emph{resp. locally isotrivial}) if for every Zariski open $U\subset X,$ the restriction $Y|_{U},$ of $Y$ on $U,$ is isomorphic to $U\times G,$ with the endowed operations $(x,g)g'=(x,gg')$ and the canonical projection $U\times G\longrightarrow U$ (\emph{resp.} if for every open $U\subset X$ there is an unramified morphism $f:U'\longrightarrow U$ over $U$ such that the inverse image $f^{-1}Y|_{U},$ of $Y|_{U},$ is trivial). $Y$ is called \emph{trivial} if $Y$ is isomorphic to $X\times G.$ 

A group $G$ is called \emph{special} if every locally isotrivial fibered system $(G,Y,X)$ is locally trivial. Finally, an isotrivial fibered system $(G,Y,X)$ is a called a \emph{$G-$principal fibration.} If, moreover, the morphism $\pi$ is flat and $(G,Y,X)$ is locally isotrivial, then $(G,Y,X)$ (or just $Y$) is a called \emph{$G-$principal bundle}.

\begin{proof}[proof of Theorem \ref{Euler-char}]

Let $\mathcal{I}(c)$ is the moduli space of rank $2$ locally free instantons on $\p3$ (without framing). We first remark that, for all $c>0,$ $\mathcal{M}_{\mathbb{P}^{3}}(c)$ is an ${\rm Sl}(2,\mathbb{C})-$bundle over $\overline{\mathcal{I}}(c),$ where the projection is given by forgetting the framing. Since the group ${\rm Sl}(2,\mathbb{C})$ is special \cite{Grothendieck1}, in the sense above, we have that every $G-$principal bundle is locally trivial in the Zariski topology. In particular $\mathcal{M}_{\mathbb{P}^{3}}(c)\to\overline{\mathcal{I}}(c)$ is a locally trivial ${\rm Sl}(2,\mathbb{C})-$principal bundle. Hence, one can write the Poincar\'e polynomial of $\mathcal{M}_{\mathbb{P}^{3}}(c)$ as$^{\ast}$ $$\mathcal{P}_{\mathcal{M}_{\mathbb{P}^{3}}(c)}(t)=\mathcal{P}_{\overline{\mathcal{I}}(c)}(t)\times\mathcal{P}_{{\rm Sl}(2,\mathbb{C})}(t),$$
and since ${\rm Sl}(2,\mathbb{C})\approxeq_{diff}{\rm SU}(2)\times\mathbb{R}^{3},$ one gets $\mathcal{P}_{{\rm Sl}(2,\mathbb{C})}(t)=1+t^{3}.$ By putting $t=-1,$ it follows that $\chi(\mathcal{M}_{\mathbb{P}^{3}}(c))=0.$

\footnotetext[1]{For the multiplicative property of the Poincar\'e polynomial the reader might see \cite[Introduction]{Brion}, for instance.}

In \cite[Section 6]{JMT2}, the authors prove that $\overline{\mathcal{I}}(1)\cong\p5.$ Hence, for $c=1$ 
the Poincar\'e polynomial is computed from the product formula.

\end{proof}

\section{Relation with multiple structures}\label{Multiple structures}

In this section we explore the relation of the rank $0$ instanton sheaves and sheaves on multiple structures \cite{Vatne, Drezet1, Drezet2, Nollet}. This allows us to give a concrete description in the lower charge cases $c=1,2,$ as well as, the multiple primitive cases (see section \ref{sheaves-multi-struc}). Moreover, we use such a description to compute the Euler Characteristic of $\mathcal{M}_{\mathbb{P}^{3}}(1).$ We also give a lower bound on the number of irreducible components.

\subsection{Monomial multiple structures}\label{multi-struc}
 
Most of the material in this subsection is borrowed from \cite{Vatne}, with the assumption that the ambient space is $\mathbb{P}^{3}.$
Let $i:X=\mathbb{P}^{1}\to\mathbb{P}^{3}$ be a linear subspace with saturated ideal $\mathcal{I}_{X},$ $X^{(i)}\subset\mathbb{P}^{3}$ the $i$'th infinitesimal neighborhood of $X,$ with ideal $(\mathcal{I}_{X})^{i+1},$ and $Y$ a Cohen-Macaulay multiple structure with $Y_{{\rm red}}=X,$ whose ideal is generated by monomials. Then the following filtration of $Y$ exists;
\begin{equation}\label{Y-filt}
X=Y_{0}\subset Y_{1}\subset \cdots \subset Y_{k-1}\subset Y_{k}=Y; \hspace{1cm} Y_{i}=Y\cap X^{(i)}, 
\end{equation}
for some $k,$ and every term $Y_{i}$ is also Cohen-Macaulay since $X$ is a Cohen-Macaulay curve \cite[Corollary 2.6]{Vatne2}. If $\mathcal{I}_{i}$ is the ideal sheaf of $Y_{i},$ then there are two short exact sequences
\begin{displaymath}
0\to\mathcal{I}_{i+1}/\mathcal{I}_{X}\mathcal{I}_{i}\to\mathcal{I}_{i}/\mathcal{I}_{X}\mathcal{I}_{i}\to\mathcal{L}_{i}\to0;
\end{displaymath}
and
\begin{displaymath}
0\to i_{\ast}\mathcal{L}_{i}\to\mathcal{O}_{Y_{i+1}}\to\mathcal{O}_{Y_{i}}\to0.
\end{displaymath}
The first exact sequence define the locally free $\mathcal{O}_{X}-$modules $\mathcal{L}_{j},$ see \cite{Nollet}, and references therein, for more details. 

One important result that will be used is the following$^{\dagger}:$

\footnotetext[2]{We only need, for our purpose, this restricted version of the more general result proved by Vatne.}

\begin{proposition}\cite[Proposition 1]{Vatne} 
There is a bijective (inclusion reversing) correspondence between Cohen-Macaulay monomial ideals in two variables and Young diagrams. Under this bijection, the number of boxes in the Young diagram is the multiplicity of the scheme defined by the corresponding ideal and whose reduced structure is a fixed line in $\mathbb{P}^{3}.$   
\end{proposition}

For instance, if we choose $\mathcal{I}_{X}=<z_2,z_3> \subset S:=\mathbb{C}[z_0,z_1,z_2,z_3],$ The Cohen-Macaulay monomial ideal $J:=<z^{3}_{2},z^{2}_{2}z^{2}_{3},z^{3}_{3}>$ will corresponds to the diagram
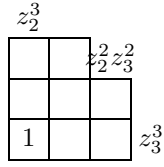
\begin{center}
\ytableausetup{mathmode}
\begin{ytableau}
\none[z^{3}_{2}] & \none & \none &\none\\
 &  & \none[z^{2}_{2}z^{2}_{3}] &\none \\
 &  &  &\none\\
1 &  & & \none[z^{3}_{3}]
\end{ytableau}
\end{center}

The number of boxes being $8,$ we have that $J$ is an ideal of a Cohen-Macaulay multiplicity $8$ structure on the line $X.$ Remark that the line $X$ itself corresponds to the box  \ytableausetup{mathmode, boxsize=1em}\begin{ytableau}1 \end{ytableau}.

\begin{definition} \hspace{10cm}
\begin{itemize}
\item An \emph{inner box} of a Young diagram will mean a box not in the diagram but such that the box bellow it and the box in its left are both in the diagram.   

\item An \emph{outer box} of a Young diagram will mean a box not in the diagram and such that the box bellow it and the box in its left are both outside the diagram, but its lower left angle touches a box in the diagram.

\end{itemize}
\end{definition}

\begin{example} In the diagram associated to $J:=<z^{3}_{2},z^{2}_{2}z^{2}_{3},z^{3}_{3}>,$ above, the red box is \emph{inner}, while the green box is \emph{outer};
\begin{center}
\ytableausetup{mathmode}
\begin{ytableau}
\none& \none & \none &\none\\
\quad &\quad  & *(red) &\none \\
\quad & \quad & \quad &\none\\
1 & \quad &\quad & \none
\end{ytableau} \hspace{2cm} \ytableausetup{mathmode}
\begin{ytableau}
\none& \none & *(green) &\none\\
 \quad&\quad  & \none &\none \\
 \quad& \quad & \quad &\none\\
1 &  \quad&\quad & \none
\end{ytableau}
\end{center}
\end{example}

\vspace{0.5cm}

\begin{proposition}\cite[Proposition 4]{Vatne} 
Given a Cohen-Macaulay monomial ideal $I$ with support on a line in $\mathbb{P}^{3},$ and its corresponding Young diagram $T.$ Then $I$ fits in the exact sequence $$0 \to\bigoplus_{j} \mathcal{O}_{\p3}(-n_{2i})\to\bigoplus_{i} \mathcal{O}_{\p3}(-n_{1i})\to I\to 0,$$ where $n_{1i}$ is the weight of the $i'$th inner box and $n_{2j}$ is the weight of the $j'$th inner box, for some chosen indexing $i,$ (\emph{resp.} $j$) of inner boxes (\emph{resp.} outer boxes)in $T.$
\end{proposition}

This way, the syzygies correspond to the outer boxes.

\begin{example} For the ideal $I$ corresponding to the Young diagram
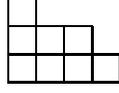
\begin{center}
\ytableausetup{mathmode}
\begin{ytableau}
\quad & \none & \none  & \none \\
\quad & \quad & \quad &\none \\
 \quad & \quad & \quad & \quad
\end{ytableau}
\end{center}
\end{example}
one has four inner boxes with weights $n_{11}=3,n_{12}=3,n_{13}=4,n_{14}=4,$ and three outer boxes with weights $n_{21}=4,n_{22}=5,n_{23}=5.$ Hence, one gets; 
$$0\to \mathcal{O}_{\p3}(-4)\oplus \mathcal{O}_{\p3}(-5)^{\oplus 2}\to \mathcal{O}_{\p3}(-4)^{\oplus 2}\oplus \mathcal{O}_{\p3}(-3)^{\oplus 2}\to I\to 0,$$

\vspace{0.5cm}

\begin{theorem}\label{multiple-class}
For a $\mathbb{T}-$fixed stable $0$ instanton pair $(\mathcal{Q},s)$ of charge $c;$
\begin{itemize}
\item[(i)] the associated scheme $\mathcal{C}$ is a multiple structure that corresponds to a Young diagram $T$ of weight $c.$ Moreover if the Young diagram is of the form$^{\ddagger}$ $\nu=(\nu_{1}\geq\nu_{2}\geq\cdots\geq\nu_{k})$ of $c,$ then $\mathcal{C}$ has Hilbert polynomial $$\chi_{\mathcal{C}}(m):=\chi(\mathcal{O}_{\mathcal{C}}(m))=cm+3c-\sum_{i=1}^{k}\frac{\nu_{i}(\nu_{i}+2i+1)}{2},$$ and $\mathcal{I}_{\mathcal{C}}$ is a smooth point in its Hilbert scheme of closed subschemes of $\p3.$ The dimension of the Hilbert scheme, of subschemes of $\p3,$ at $\mathcal{I}_{\mathcal{C}}$ is given by
\begin{align}
D_{\mathcal{I}_{\mathcal{C}}}=&\sum_{n_{2j}\geq n_{1i}}\binom{n_{2j}-n_{1i}+3}{3}+\sum_{n_{1i}\geq n_{2j}}\binom{n_{1i}-n_{2j}+3}{3} \notag \\
&-\sum_{n_{2j}\geq n_{2i}}\binom{n_{2j}-n_{2i}+3}{3}-\sum_{n_{1j}\geq n_{1i}}\binom{n_{1j}-n_{1i}+3}{3}+1. \notag
\end{align}
\item[(ii)] Moreover, the associated sheaf $\mathcal{Z}:=\coker(s)$ has length $$l_{\mathcal{Z}}=\frac{1}{2}\sum_{i=1}^{k}\nu^{2}_{i}+\sum_{i=1}^{k}i\nu_{i}-\frac{c}{2},$$ where $\nu=(\nu_{1}\geq\nu_{2}\geq\cdots\geq\nu_{k})$ is the partition of $c,$ represented by the Young diagram $T$ associated to the multiple structure $\mathcal{C}.$
\end{itemize}
\end{theorem}

\footnotetext[3]{The Young diagram is associated to a partition $\nu$ of $c,$ where the $i'$th column represents the $i'$th part $\nu_{i},$ $i=1,\cdots,k.$  }

\begin{proof}
\begin{itemize}
\item[(i)]For a $\mathbb{T}-$fixed stable $0$ instanton pair $(\mathcal{Q},s)$ the schematic support $\mathcal{C}$ of $\mathcal{Q}$ should also be invariant. By Theorem \ref{strict-torsion-free}, it follows that it is a multiple structure on the unique line that does not intersect the framing line in $\p3.$ Hence its ideal $\mathcal{I}_{\mathcal{C}}$ should be generated by monomials.
The Hilbert polynomial $\chi_{\mathcal{C}}(m)$ can be computed according to \cite[Corollary 2]{Vatne} and using the fact that the weight of a box $(i,j)$ is given by $w_{i,j}=i+j-2.$ The dimension $D_{\mathcal{I}_{\mathcal{C}}},$ of the Hilbert scheme of subschemes of $\p3$ at $\mathcal{C},$ follows from \cite[Corollary 1]{Vatne}.
\item[(ii)] The statement about $\mathcal{Z}$ follows from item {\rm (i)} and the exact sequence $$0\to\mathcal{O}_{\mathcal{C}}\to\mathcal{Q}\to\mathcal{Z}\to0.$$
\end{itemize}

\end{proof}

The above result classifies all scheme theoretic supports of the $\mathbb{T}-$fixed stable $0$ instanton pair $(\mathcal{Q},s).$

%---------------------------------------------------------------

\subsection{Sheaves on multiple structures}\label{sheaves-multi-struc}
After classifying the possible schematic supports of the pair, we will now study the sheaf $\mathcal{Q},$ emanating from the $\mathbb{T}-$fixed stable pair $(\mathcal{Q},s),$ as a sheaf on the monomial double structure $\mathcal{C}$ defined over the line $l_{0}=(z_{2}=z_{3}=0).$ For instance, we give a complete description of $\mathcal{Q}$ for small charges, namely $c=1,2$ and $3.$ In order to achieve this goal we first recall some results from \cite{Drezet1}, \cite{Drezet2}.

For $X=l_{0}\subset Y\subset \p3,$ as in $\S$ \ref{multi-struc}, with a filtration \eqref{Y-filt} we say that $Y$ is \emph{primitive} if for every $x\in X,$ there exists a surface $S$ of $\p3$ which is smooth at $x$ and containing a neighborhood of $x$ in $Y.$ In this case, $L=\mathcal{I}_{X}/\mathcal{I}_{Y_{2}}$ is an invertible sheaf on $X$ and we have $\mathcal{I}_{Y_{i}}/\mathcal{I}_{Y_{i+1}}=L^{i}$ for $1\leq i\leq c.$ This means that for a point $x\in X,$ there are elements $z_{2},z_{3},t,$ of the maximal ideal $m_{X,x}$ of $x$ in $\mathcal{O}_{X,x},$ such that their images in $m_{X,x}/m^{2}_{X,x}$ form a basis and for all $1\leq i\leq c$ one has $\mathcal{I}_{Y_{i},x}=<z_{2},z^{i}_{3}>.$
Let $\mathcal{F}$ be a coherent sheaf over $Y$
\begin{definition}
The {\rm first canonical filtration} of $\mathcal{F}$ is the filtration
$$\mathcal{F}_{c+1}=0\subset\mathcal{F}_{c}\subset\cdots\subset\mathcal{F}_{2}\subset\mathcal{F}_{1}=\mathcal{F};$$ such that, for $1\leq i\leq c,$ $\mathcal{F}_{i+1}$ is inductively defined as the kernel of the restriction morphism $\mathcal{F}_{i}\to\mathcal{F}_{i}|_{X}.$

\end{definition}
In this way one has $\mathcal{F}_{i}/\mathcal{F}_{i+1}=\mathcal{F}_{i}|_{X}$ and $\mathcal{F}/\mathcal{F}_{i+1}=\mathcal{F}|_{Y_{i}}.$ The graded object ${\rm Gr}(\mathcal{F})=\bigoplus_{i=1}^{c}\mathcal{F}_{i}/\mathcal{F}_{i+1}$ is then an $\mathcal{O}_{X}-$module$^{\dagger\dagger}.$
\footnotetext[8]{A \emph{second canonical filtration}, that we won't use, is also defined in \cite[\S 4]{Drezet1}. The interested reader might check the given reference.}

Some properties of these filtrations can be listed as it follows \cite[\S 3]{Drezet2};
\begin{itemize}
    \item For the ideal $\mathcal{I}_{X},$ of $X,$ in $\mathcal{O}_{Y}$ and a coherent sheaf $\mathcal{F},$ over $Y,$ one has $\mathcal{F}_{i}=\mathcal{I}_{X}^{i}\mathcal{F}$ so that ${\rm Gr}(\mathcal{F})=\bigoplus_{i=0}^{c-1}\mathcal{I}_{X}^{i}\mathcal{F}/\mathcal{I}_{X}^{i+1}\mathcal{F};$
    
    \item $\mathcal{F}_{i}=0$ if and only if $\mathcal{F}$ is a sheaf over $Y_{i};$
    
    \item for each $0\leq i\leq c,$ $\mathcal{F}_{i}$ is a coherent sheaf over $Y_{i}$ with first canonical filtration $0\subset\mathcal{F}_{c}\subset\cdots\subset\mathcal{F}_{i+1}\subset\mathcal{F}_{i};$
    
    \item  Morphisms of coherent sheaves $\mathcal{F}\to\mathcal{G},$ on $Y,$ induce morphisms of first canonical filtrations $\mathcal{F}_{i}\to\mathcal{G}_{i},$ for all $0\leq i\leq c,$ and hence induce morphisms of the graded objects ${\rm Gr}(\mathcal{F})\to{\rm Gr}(\mathcal{G}).$
\end{itemize}

\begin{definition} \hspace{10cm} 
\begin{itemize}
\item {\rm The generalised rank} is defined by the integer $R(\mathcal{F})=rk({\rm Gr}(\mathcal{F})).$
\item {\rm The generalised degree} is defined by the integer $Deg(\mathcal{F})=deg({\rm Gr}(\mathcal{F})).$
\end{itemize}
\end{definition}

The generalised rank and degree are defined so that they behave additively on exact sequence on $Y.$ In general the usual rank and degree fail to satisfy this condition. Moreover we have the following generalised Riemann-Roch Theorem:
\begin{theorem}\cite[{\bf Theorem 4.2.1}]{Drezet1}\label{Euler}
For a coherent sheaf $\mathcal{F},$ over $Y,$ we have $$\chi(\mathcal{F})=Deg(\mathcal{F})+R(\mathcal{F})(1-g_{Y}).$$
\end{theorem}
Here, $g_{Y}$ is the genus of the curve $Y.$

\vspace{0.5cm}

\subsubsection{Stable rank $0$ instanton pair of charge $1$}\label{charge1}

In this case the only possible support is the line $l_{0},$ the line that does not intersect the framing line $l_{\infty}.$ 
The sheaf $\mathcal{Q}$ sits in the short exact sequence $$0\to\mathcal{O}_{l_{0}}\to\mathcal{Q} \to\mathcal{Z}\to0,$$ where $\mathcal{Z}$ is the structure sheaf of one point. Hence the only possibilities are $\mathcal{Q}=\mathcal{O}_{l_{0}}(p_{i}),$ $i=0,1$ Where $p_{0}=[1;0;0;0]$ and $p_{1}=[0;1;0;0].$ We point out that the rank $2$ fixed instanton bundles given by $\ker(\mathcal{O}_{\p3}^{2}\to\mathcal{O}_{l_{0}}(p_{i}))$ are nullcorrelation sheaves \cite{Ein}. Moreover, we see that these $\mathbb{T}-$fixed points are isolated.

\begin{corollary}
The moduli $\mathcal{M}_{\mathbb{P}^{3}}(1)$ has only one fixed point under the lifted toric action on $\p3.$ 
\end{corollary}

\begin{proof}
Since the $\mathbb{T}-$fixed $0-$rank instantons can only be  $\mathcal{O}_{l_{0}}(p_{i}),$ $i=0,1,$ one gets two quotient maps $$\mathcal{O}_{\p3}^{\oplus2}\twoheadrightarrow\mathcal{O}_{l_{0}}(p_{i})\quad i=0,1,$$ with isomorphic kernels $E_{i},$ $i=0,1.$ These are fixed framed instanton sheaves and all of them belong to the the same class $E.$  Hence in this case there is only one isolated point.

\end{proof}

The next result gives the tangent and obstruction spaces at the fixed $0$ instanton pairs 

\begin{lemma}
For the stable pairs $\rho_{i}=(\mathcal{Q}_{i}=\mathcal{O}_{l_{0}}(p_{i}),s)\in P_{1}(\mathbb{P}^{3},\beta=H^{2})^{\mathbb{T}},$ $i=0,1,$ one has
$${\rm T}_{\rho_{i}}P_{1}(\mathbb{P}^{3},\beta=H^{2})^{\mathbb{T}}=\mathbb{C}^{5}, \qquad {\rm Obs}_{\rho_{i}}P_{1}(\mathbb{P}^{3},\beta=H^{2})^{\mathbb{T}}=\mathbb{C}^{3}.$$ Moreover 
\end{lemma}

\begin{proof}

Recall from \cite{Panda-Thom} that we have a triangle 

\begin{align}
    \mathcal{Q}_{i}[-1] \to I^{\bullet}\to\mathcal{O}_{\p3}\to\mathcal{Q}_{i} \label{Tri1}
\end{align}
 
in $D^{b}(\p3),$ where $\mathcal{Q}=\mathcal{O}_{l_{0}}(p_{i}),$ $i=1,2,$ and $I^{\bullet}:=\{\mathcal{O}_{\p3}\to\mathcal{Q}\}.$

Applying $\Hom(-,\mathcal{Q})$ on \eqref{Tri1}, one gets the sequence

\begin{equation}\label{Long-Triangle}
    \xymatrix@C-1.2pc@R-2pc{
    \ext^{-1}(I^{\bullet},\mathcal{O}_{l_{0}}(p_{i}))\ar[r] & \End(\mathcal{O}_{l_{0}}(p_{i}))\ar[r] & \ext^{0}(\mathcal{O}_{\p3},\mathcal{O}_{l_{0}}(p_{i}))\ar[r]& \\ 
    \ext^{0}(I^{\bullet},\mathcal{O}_{l_{0}}(p_{i}))\ar[r]& \ext^{1}(\mathcal{O}_{l_{0}}(p_{i}),\mathcal{O}_{l_{0}}(p_{i}))\ar[r]& \ext^{1}(\mathcal{O}_{\p3},\mathcal{O}_{l_{0}}(p_{i}))\ar[r]& \\
    \ext^{1}(I^{\bullet},\mathcal{O}_{l_{0}}(p_{i}))\ar[r]&
    \ext^{2}(\mathcal{O}_{l_{0}}(p_{i}),\mathcal{O}_{l_{0}}(p_{i}))&&
    }
\end{equation}

where $\ext^{-1}(I^{\bullet},\mathcal{O}_{l_{0}}(p_{i}))=0,$ as in the proof of \cite[Lemma 1.5]{Panda-Thom}. Observe that $\chi(I^{\bullet},\mathcal{O}_{l_{0}}(p_{i}))=2,$ and $\chi(\mathcal{O}_{l_{0}}(p_{i}),\mathcal{O}_{l_{0}}(p_{i}))=4,$ by Hirzebruch-Riemann-Roch Theorem. One can easily compute that
{\small $$\End(\mathcal{O}_{l_{0}}(p_{i}))=\mathbb{C}, \hspace{0.3cm} \ext^{1}(\mathcal{O}_{l_{0}}(p_{i}),\mathcal{O}_{l_{0}}(p_{i}))=\mathbb{C}^{4}, \hspace{0.3cm} \ext^{2}(\mathcal{O}_{l_{0}}(p_{i}),\mathcal{O}_{l_{0}}(p_{i}))=\mathbb{C}^{3},$$} and also  

{\small
$$\ext^{1}(\mathcal{O}_{\p3},\mathcal{O}_{l_{0}}(p_{i}))\cong\ho^{1}(\mathcal{O}_{l_{0}}(p_{i}))=0,\hspace{0.3cm}\Hom(\mathcal{O}_{\p3},\mathcal{O}_{l_{0}}(p_{i}))\cong\ho^{0}(\mathcal{O}_{l_{0}}(p_{i}))=\mathbb{C}.$$ }
 
From \eqref{Long-Triangle}, it follows that $$\ext^{1}(I^{\bullet},\mathcal{O}_{l_{0}}(p_{i}))\cong\ext^{2}(\mathcal{O}_{l_{0}}(p_{i}),\mathcal{O}_{l_{0}}(p_{i}))\cong\mathbb{C}^{3}\textnormal{ and } 
\ext^{0}(I^{\bullet},\mathcal{O}_{l_{0}}(p_{i}))\cong\mathbb{C}^{5}$$

\end{proof}

The dimension of the tangent space at the pair $(\mathcal{Q},s)$ is $5$ as expected, since one has to choose a line in $\p3,$ thus a point in $\mathbb{G}(2,4),$ and a section in $\mathbb{P}^{1}=\mathbb{P}(\ho^{0}(\mathcal{Q})).$ Moreover, the dimension of the obstruction space does not jump, and hence the fixed locus is, in this case, smooth.

\subsubsection{Stable rank $0$ instanton pair of charge $2$}
For a $\mathbb{T}-$fixed stable rank $0$ instanton pair $(\mathcal{Q},s)$ of charge $2,$ the associated Cohen-Macaulay curve $\mathcal{C}$ is a primitive double curve with ideal generated by monomials, hence one can associate to it one of the following Young digrams:
\begin{center}
\ytableausetup{mathmode}
\begin{ytableau}
 \none[z_{2}] & \none & \none \\
 \quad & \quad & \none[z_{3}^{2}]
\end{ytableau} \hspace{2cm}
\ytableausetup{mathmode}
\begin{ytableau}
\none[z_{2}^{2}] & \none \\
\quad & \none \\
\quad & \none[z_{3}]
\end{ytableau}
\end{center}

We will only treat the case \ytableausetup{mathmode}
\begin{ytableau}
 \quad & \quad  
\end{ytableau}, the other case being very similar;
The ideal sheaf of $\mathcal{I}_{\mathcal{C}},$ of $\mathcal{C},$ in $\mathcal{O}_{\p3}$ is $\mathcal{I}_{\mathcal{C}}=<z_{2},z_{3}^{2}>,$ and $\mathcal{C}$ is clearly a complete intersection. Moreover, it is easy to see that we have $$0 \to\mathcal{O}_{\p3}(-3)\to \mathcal{O}_{\p3}(-2)\oplus \mathcal{O}_{\p3}(-1)\to \mathcal{I}_{\mathcal{C}}\to 0,$$ with Hilbert polynomial $\chi(m)=2m-1,$ so that $l_{\mathcal{Z}}=3.$ Using \cite[Lemma 1.3]{Nollet} one has a sequence$^{}$ $$0\to\mathcal{I}_{\mathcal{C}}\to\mathcal{I}_{l_{0}}\to L\cong\mathcal{O}_{l_{0}}(-1)\to0,$$ hence the restriction sequence $$0\to\mathcal{O}_{l_{0}}(-1)\to\mathcal{O}_{\mathcal{C}}\to \mathcal{O}_{l_{0}}\to0.$$
Thus the first canonical filtration of $\mathcal{O}_{\mathcal{C}}$ is simply $0\subset\mathcal{O}_{l_{0}}(-1)\subset\mathcal{O}_{\mathcal{C}},$ and the graded sheaf associated to it is ${\rm Gr}(\mathcal{O}_{\mathcal{C}})=\mathcal{O}_{l_{0}}\oplus\mathcal{O}_{l_{0}}(-1).$ This gives the generalised rank and degree, respectively, $R(\mathcal{O}_{\mathcal{C}})=2,$ $Deg(\mathcal{O}_{\mathcal{C}})=-1.$

\vspace{0.3cm}

$\mathcal{Q}$ has first canonical filtration $0\subset\mathcal{Q}_{2}\subset\mathcal{Q}$ with a graded object ${\rm Gr}(\mathcal{Q})=\mathcal{Q}|_{l_{0}}\oplus\mathcal{Q}_{2}.$ Thus, one obtains the diagram

\begin{equation}\label{rest-filt}
\xymatrix@R-1.3pc@C-1.3pc{
& 0\ar[d]& 0 \ar[d]& 0 \ar[d] & \\
0\ar[r]& \mathcal{O}_{l_{0}}(-1)\ar[d]\ar[r] & \mathcal{O}_{\mathcal{C}}\ar[r]\ar[d]& \mathcal{O}_{l_{0}}\ar[r]\ar[d] &0 \\
0\ar[r] & \mathcal{Q}_{2}\ar[d]\ar[r] & \mathcal{Q} \ar[d]\ar[r]& \mathcal{Q}|_{l_{0}} \ar[r]\ar[d]&0 \\
0\ar[r]& \mathcal{Z}_{1}\ar[r]\ar[d]& \mathcal{Z}\ar[r] \ar[d]& \mathcal{Z}_{2}\ar[d]\ar[r]&0 \\
&0&0&0&
}
\end{equation}
Thus $\mathcal{Q}_{2}=\mathcal{O}_{l_{0}}(-1),$ the generalised rank and degree of $\mathcal{Q}$ are, respectively,  $R(\mathcal{Q})=2$ and $Deg(\mathcal{Q})=2.$ This leaves us with the following possibility:

\begin{theorem}\label{type}
$\mathcal{Q}|_{l_{0}}\cong\mathcal{O}_{l_{0}}(1)\oplus T,$ where $T$ is a torsion sheaf of length $1.$
\end{theorem}

\begin{proof}
Torsion free sheaves of generalised rank $2$ on the double line are of three types \cite[\S 8.2]{Drezet1}, namely line bundles, vector bundles on $l_{0}$ and the strictly Torsion-free;

\vspace{0.3cm}

If $\mathcal{Q}$ is a vector bundle on $l_{0},$ then it is equal to its restriction, which contradicts the diagram \eqref{rest-filt} by the fact that $\mathcal{Q}_{2}=0,$ and the snake lemma, implies that $ \mathcal{Z}_{1}$ is a pure torsion subsheaf of $\mathcal{O}_{l_{0}}.$

\vspace{0.3cm}

If $\mathcal{Q}$ is a line bundle on $\mathcal{C},$ then its restriction is the line bundle $\mathcal{O}_{l_{0}}(3),$ which is the only possibility compatible with the right column in \eqref{rest-filt}. On the other hand $\mathcal{Q}$ fits in the following exact sequence
$$0\to D\otimes\mathcal{O}_{l_{0}}(-1)\to\mathcal{Q}\to D\to0$$
where $D=\mathcal{O}_{l_{0}}(3).$ But this means that $\mathcal{O}_{l_{0}}(3)\otimes\mathcal{O}_{l_{0}}(-1)\cong\mathcal{O}_{l_{0}}(-1).$ Hence, $\mathcal{Q}$ cannot be a line bundle on $\mathcal{C}.$

\vspace{0.3cm}

Finally, in a more general situation $\mathcal{Q}$ fits in a short exact sequence 

$$0\to D\otimes\mathcal{O}_{l_{0}}(-1)\to\mathcal{Q}\to D\oplus T\to0,$$
where $D$ is a line bundle on $l_{0}$ and $T$ is a torsion sheaf, also on $l_{0}.$ Twisting the diagram by $\mathcal{O}_{\mathbb{P}^{3}}(-2)$ and using the vanishing conditions $\ho^{0,1}(\mathcal{Q}(-2))=0,$ one has that $0<d:=deg(D)<3$ and $\ho^{1}(D(-3))=\ho^{0}(D(-2))\oplus\ho^{0}(T).$ 

If $d=2$ then $\ho^{1}(D(-3))=\ho^{1}(\mathcal{O}_{l_{0}}(-1))=0\ho^{0}(T).$ It follows that the torsion sheaf $T$ is zero, which is not possible, as the restriction $\mathcal{Q}|_{l_{0}}$ is not locally free.  

If  $d=1$ then $\ho^{1}(D(-3))=\ho^{1}(\mathcal{O}_{l_{0}})=\mathbb{C}=\ho^{0}(T).$ Thus $T$ is a torsion sheaf of length $1,$ and it follows, from display \eqref{rest-filt}, that $\mathcal{Z}_{2}$ is the structure sheaf of $2$ points.

%One can compare this sequence to the middle row in \eqref{rest-filt}. Then, by calculating the generalised rank and degree, one finds that the only allowed case is $D=\mathcal{O}_{l_{0}}$ and $T=\mathcal{Z}.$

\end{proof}

\begin{corollary}\label{componts}
$\mathcal{M}^{\mathbb{T}}_{\mathbb{P}^{3}}(2)$ has at least $2$ irreducible component. 
\end{corollary}

\vspace{0.4cm}

\begin{proof}
From Theorem \ref{multiple-class}, and as we saw in the beginning of this section, there are two possible Young diagrams for the support. To each one of these curves there is one possible filtration $0\subset\mathcal{O}_{l_{0}}\subset\mathcal{Q}$ with restriction $\mathcal{O}_{l_{0}}(1)\oplus T,$ where length$(T)=1$. Finally, by Theorem \ref{type} and the fact that $\mathcal{Q}$ might have non trivial deformation in each case, then the result follow. 

\end{proof}

We remark that the double curve $\mathcal{C}$ can be deformed into two lines intersecting in a point, but we don't know if one can deform the $0-$rank instanton sheaves $\mathcal{Q}$ into torsion free sheaves on the reducible curve formed by two intersecting lines. This is a hard problem and should be investigated in the future.

\vspace{0.3cm}

\subsubsection{Stable rank $0$ instanton pair of charge $3$}

In this case one has $3$ possible associated Young diagrams, namely

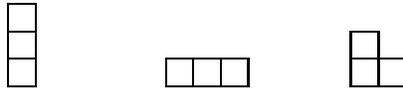
\begin{center}
\ytableausetup{mathmode}
\begin{ytableau}
\quad & \none \\
\quad &  \none  \\
 \quad &  \none 
\end{ytableau}  \hspace{1cm}
\ytableausetup{mathmode}
\begin{ytableau}
\none & \none \\
\none &  \none  \\
 \quad & \quad & \quad 
\end{ytableau} \hspace{1cm}
\ytableausetup{mathmode}
\begin{ytableau}
\none & \none \\
\quad & \none \\
 \quad & \quad 
\end{ytableau}
\end{center}

In the following, we will treat only the primitive cases, so we can concentrate on the horizontal Young diagram, the vertical case being similar. The canonical filtration of the triple curve $\mathcal{C}$ is given by $0\subset\mathcal{L}_{3}=\mathcal{O}_{l_{0}}(-1)\subset\mathcal{L}_{2}\subset\mathcal{O}_{\mathcal{C}}$ with quotients $\mathcal{O}_{\mathcal{C}}/\mathcal{L}_{2}=\mathcal{O}_{l_{0}},$ $\mathcal{O}_{\mathcal{C}}/\mathcal{L}_{3}=\mathcal{O}_{\mathcal{C}_{2}}$ and  $\mathcal{L}_{2}/\mathcal{L}_{3}=\mathcal{O}_{l_{0}}(-1).$ we also recall that $\chi(\mathcal{O}_{\mathcal{C}}(m))=3m$ and $\chi(\mathcal{Q}(m))=3m+6.$ On the other hand one has a canonical filtration $0\subset\mathcal{Q}_{3}\subset\mathcal{Q}_{2}\subset\mathcal{Q}.$ Twisting by $\mathcal{O}_{\mathbb{P}^{3}}(-2)$ and using the instanton conditions $\ho^{0,1}(\mathcal{Q}(-2))=0,$ one has

\begin{equation}\label{cascade:cond}
\begin{array}{ll}
(a) &\ho^{0}(\mathcal{Q}_{2}(-2))=\ho^{0}(\mathcal{Q}_{3}(-2))=0, \\
(b) &\ho^{1}(\mathcal{Q}|_{l_{0}}(-2))=\ho^{1}(\mathcal{Q}|_{\mathcal{C}_{2}}(-2))=0, \\
(c) &\ho^{1}(\mathcal{Q}_{2}(-2))=\ho^{0}(\mathcal{Q}|_{l_{0}}(-2)) \textnormal{ and } \\
(d) &\ho^{1}(\mathcal{Q}_{3}(-2))=\ho^{0}(\mathcal{Q}|_{\mathcal{C}_{2}}(-2))
\end{array}
\end{equation}

Moreover, $\mathcal{Q}_{2}(-2)$ is a generalised rank $2$ sheaf on $\mathcal{C}_{2}$ and $\mathcal{Q}_{3}(-2)$ is a line bundle on $l_{0}$. Then we have the following exact sequence, associated to the canonical filtration of $\mathcal{Q}_{2}(-2),$
\begin{displaymath}
0\to D\otimes \mathcal{O}_{l_{0}}(-1)\to\mathcal{Q}_{2}(-2)\to D\oplus T_{2}\to0,
\end{displaymath}
in which $D=\mathcal{Q}_{3}(-1)$ and $T_{2}$ is pure torsion sheaf on $l_{0}.$ But from \eqref{cascade:cond}, $(a)$ and $(b),$ the degree $d,$ of  $D=\mathcal{Q}_{3}(-1),$ satisfies $-2<d<1,$ i. e., $deg(\mathcal{Q}_{3})=0,$ or $deg(\mathcal{Q}_{3})=1.$  

The next step is to consider the commutative diagrams

\begin{equation}\label{level:2}
\xymatrix@R-1pc@C-1pc{
& 0\ar[d]& 0 \ar[d]& 0 \ar[d] & \\
0\ar[r]& \mathcal{L}_{2}(-2)\ar[d]\ar[r] & \mathcal{O}_{\mathcal{C}}(-2) \ar[r]\ar[d]& \mathcal{O}_{l_{0}}(-2)\ar[r]\ar[d] &0 \\
0\ar[r] & \mathcal{Q}_{2}(-2)\ar[r] \ar[d]& \mathcal{Q}(-2) \ar[d]\ar[r]&  \mathcal{Q}|_{l_{0}}(-2) \ar[r]\ar[d]&0\\
0\ar[r]&   \bar{\mathcal{Z}}\ar[d]\ar[r]& \mathcal{Z}_{6}\ar[r] \ar[d]& \tilde{\mathcal{Z}}\ar[d]\ar[r]& 0\\
&0&0&0&
}
\end{equation}

\begin{equation}\label{level:3}
\xymatrix@R-1pc@C-1pc{
& 0\ar[d]& 0 \ar[d]& 0 \ar[d] & \\
0\ar[r]& \mathcal{O}_{l_{0}}(-4)\ar[d]\ar[r] & \mathcal{L}_{2}(-2)\ar[r]\ar[d]& \mathcal{O}_{l_{0}}(-3)\ar[r]\ar[d] &0 \\
0\ar[r] & \mathcal{Q}_{3}(-2)\ar[r] \ar[d]& \mathcal{Q}_{2}(-2) \ar[d]\ar[r]&  \mathcal{Q}_{3}(-1)\oplus T_{2} \ar[r]\ar[d]&0\\
0\ar[r]&  \mathcal{Z}_{2}\ar[d]\ar[r]& \bar{\mathcal{Z}}\ar[r] \ar[d]& \mathcal{Z}_{2}\oplus T_{2}\ar[d]\ar[r]&0 \\
&0&0&0&
}
\end{equation}

associated to the sequence $$0\to\mathcal{O}_{\mathcal{C}}\to\mathcal{Q}\to\mathcal{Z}_{6}\to0$$ and the canonical filtrations of $\mathcal{O}_{\mathcal{C}}$ and $\mathcal{Q}.$ Here we remind the reader that $\mathcal{Z}_{6}$ has length $6$  and $\mathcal{Z}_{2}$ is a torsion sheaf of length $2,$ if $\mathcal{Q}_{3}= \mathcal{O}_{l_{0}}$ or $3,$ if $\mathcal{Q}_{3}= \mathcal{O}_{l_{0}}(1),$ as it clearly appears from left column of diagram \eqref{level:3}.

Suppose that $\mathcal{Q}_{3}=\mathcal{O}_{l_{0}}(1),$ then, from the middle row of diagram \eqref{level:3}, one has $\chi(\mathcal{Q}_{2}(-2))=t_{2}+1.$ In particular, $\chi(\mathcal{Q}_{2}(-2))>0,$ since $t_{2}\geq0.$ On the other hand, we have  $\ho^{0}(\mathcal{Q}_{2}(-2))=0,$ and it follows that $\chi(\mathcal{Q}_{2}(-2))=-{\rm dim}\ho^{1}(\mathcal{Q}_{2}(-2))<0.$ Thus $\mathcal{Q}_{3}$ cannot be $\mathcal{O}_{l_{0}}(1),$ and we are left with $\mathcal{Q}_{3}=\mathcal{O}_{l_{0}}.$ In this case one has $\chi(\mathcal{Q}_{2}(-2))=t_{2}-1$ and $t_{2}$ must be $0$ or $1,$ since  $\chi(\mathcal{Q}_{2}(-2))\leq0.$ 

If $t_{2}=0,$ then $T_{2}=0,$ $\mathcal{Q}_{2}(-2)=\mathcal{L}_{2},$ since the middle row of \eqref{level:3} is exactly the restriction sequence, given by the canonical filtration, of $\mathcal{L}_{2}$ to $l_{0}.$ Furthermore, from the lower rows of \eqref{level:2} and \eqref{level:3}. It follows that $\bar{\mathcal{Z}}$ has length $4,$ thus $\tilde{\mathcal{Z}}$ has length $2.$ By using the right column of \eqref{level:2} one has $\chi(\mathcal{Q}|_{l_{0}}(-2))=1,$ Hence $\mathcal{Q}|_{l_{0}}=\mathcal{O}_{l_{0}}.$
%remark that $\ho^{0}(\mathcal{L}_{2})=0,$ $\ho^{1}(\mathcal{L}_{2})=\mathbb{C}=\ho^{0}(\mathcal{O}_{l_{0}})$ and $\ho^{1}(\mathcal{O}_{l_{0}})=0.$ which is compatible with conditions \eqref{cascade:cond}.

Finally if $t_{2}=1$ then, from the middle row of \eqref{level:3}, one can see that $\mathcal{Q}_{2}$ also satisfies $\ho^{0,1}(\mathcal{Q}_{2}(-2))=0,$ hence it is an rank $0$ instanton sheaf over $\mathcal{C}_{2}.$ Moreover, the length of $\bar{\mathcal{Z}}$ is equal to $5$ and it follows, again, from the right column of diagram \eqref{level:3} that $\tilde{\mathcal{Z}}$ has length $1$ and $\mathcal{Q}|_{l_{0}}=\mathcal{O}_{l_{0}}(1).$ This proves the following

\begin{theorem}\label{cases:c=3}
For $c=3$ and $\mathcal{C}$ is primitive and monomial, the sheaf $\mathcal{Q}$ has a canonical filtration $0\subset\mathcal{Q}_{3}=\mathcal{O}_{l_{0}}\subset\mathcal{Q}_{2}\subset\mathcal{Q}$ in which  

\begin{itemize}
\item[(i)] $\mathcal{Q}_{2}=\mathcal{L}_{2}(2),$ $\mathcal{Q}|_{l_{0}}=\mathcal{O}_{l_{0}}(2),$ and $\mathcal{Q}_{2}/\mathcal{Q}_{3}=\mathcal{O}_{l_{0}}(1),$ or

\item[(ii)] $\mathcal{Q}_{2}$ is a rank $0$ instanton sheaf on $\mathcal{C}_{2}\subset\mathcal{C},$ $\mathcal{Q}|_{l_{0}}=\mathcal{O}_{l_{0}}(1),$ and $\mathcal{Q}_{2}/\mathcal{Q}_{3}=\mathcal{O}_{l_{0}}(1)\oplus T_{2},$ where $T_{2}$ is a torsion sheaf of length $1.$
\end{itemize}

\end{theorem}

Now we turn our attention to the first non primitive case, that is, when the corresponding Young diagram is
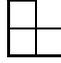
\begin{center}
 \ytableausetup{mathmode}
\begin{ytableau}
\none & \none \\
\quad & \none \\
 \quad & \quad 
\end{ytableau}
\end{center}
This is the case of the (affine) ideal $<x,y>^{2}.$ It is easy to check that the restriction map is given by:
$$0\to\mathcal{O}_{l_{0}}(-1)^{\oplus 2}\to\mathcal{O}_{\mathcal{C}}\to\mathcal{O}_{l_{0}}\to0.$$ There are two filtrations represented as below:

{\tiny
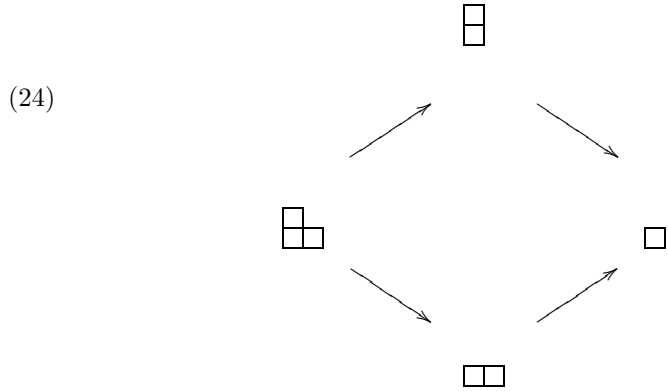
\begin{center}
\ytableausetup{mathmode}
\begin{ytableau}
\quad &  \none  \\
 \quad &  \none 
\end{ytableau} 
\end{center}}
\begin{equation}\label{filts}\xymatrix@R-0.5pc@C-1pc{ & & \\ \ar[rru] & & }\hspace{1cm}\xymatrix@R-0.5pc@C-1pc{  \ar[rrd] & & \\ & & }\end{equation} {\tiny
\begin{center}
\ytableausetup{mathmode}
\begin{ytableau}
\none & \none \\
\quad & \none \\
 \quad & \quad 
\end{ytableau} \hspace{4cm}
\ytableausetup{mathmode}
\begin{ytableau}
\none& \none\\
\none &\none \\
\quad &  \none
\end{ytableau} 
\end{center}
$$\xymatrix@R-0.5pc@C-1pc{  \ar[rrd] & & \\& & }\hspace{1cm}\xymatrix@R-0.5pc@C-1pc{ & & \\  \ar[rru] & & }$$
\begin{center}
\ytableausetup{mathmode}
\begin{ytableau}
\none &  \none  \\
 \quad & \quad  
\end{ytableau} 
\end{center}
}

Note that the double structures in the middle column are primitive, and although there is no unique canonical filtration, we still manage to compute the resulting possible pure sheaves $\mathcal{Q}.$ We use for instance the filtration given by the upper arrow of \eqref{filts}, and as in the previous theorem, we consider restriction diagrams as \eqref{level:2} and \eqref{level:3};

\begin{equation}\label{level:2-2}
\xymatrix@R-1pc@C-1pc{
& 0\ar[d]& 0 \ar[d]& 0 \ar[d] & \\
0\ar[r]& \mathcal{O}_{l_{0}}(-3)^{\oplus 2}\ar[d]\ar[r] & \mathcal{O}_{\mathcal{C}}(-2) \ar[r]\ar[d]& \mathcal{O}_{l_{0}}(-2)\ar[r]\ar[d] &0 \\
0\ar[r] & \mathcal{Q}_{2}(-2)\ar[r] \ar[d]& \mathcal{Q}(-2) \ar[d]\ar[r]&  \mathcal{Q}|_{l_{0}}(-2) \ar[r]\ar[d]&0\\
0\ar[r]&   \bar{\mathcal{Z}}\ar[d]\ar[r]& \mathcal{Z}_{5}\ar[r] \ar[d]& \tilde{\mathcal{Z}}\ar[d]\ar[r]& 0\\
&0&0&0&
}
\end{equation}

\begin{equation}\label{level:3-2}
\xymatrix@R-1pc@C-1pc{
& 0\ar[d]& 0 \ar[d]& 0 \ar[d] & \\
0\ar[r]& \mathcal{O}_{l_{0}}(-3)\ar[d]\ar[r] & \mathcal{O}_{\mathcal{C}}(-2) \ar[r]\ar[d]& \mathcal{O}_{\mathcal{C}_{2}}(-2) \ar[r]\ar[d] &0 \\
0\ar[r] & \mathcal{Q}_{3}(-2)\ar[r] \ar[d]& \mathcal{Q}(-2) \ar[d]\ar[r]&  \mathcal{Q}|_{\mathcal{C}_{2}}(-2) \ar[r]\ar[d]&0\\
0\ar[r]&  \hat{\mathcal{Z}}_{2}\ar[d]\ar[r]& \mathcal{Z}_{5}\ar[r] \ar[d]& \check{\mathcal{Z}}_{2}\ar[d]\ar[r]&0 \\
&0&0&0&
}
\end{equation}

We recall that $\chi(\mathcal{Q}(m))=3m+6,$ $\chi(\mathcal{O}_{\mathcal{C}}(m))=3m+1,$ so that $\mathcal{Z}_{5}$ has length $5.$

\begin{theorem}\label{non-primitive}
For $c=3$ and $\mathcal{C}$ the non primitive monomial curve, the sheaf $\mathcal{Q}$ has a filtration $0\subset\mathcal{O}_{3}\subset\mathcal{Q}_{2}\subset\mathcal{Q}$ such that  

\begin{itemize}
\item[(i)] $\mathcal{Q}|_{l_{0}}=\mathcal{O}_{l_{0}}(1)\oplus T_{3},$ where $T_{3}$ is a torsion sheaf of length $3,$ $\mathcal{Q}_{2}$ is a sheaf on $\mathcal{C}_{2}$ with restriction sequence $$0\to\mathcal{O}_{l_{0}}(-1)\to\mathcal{Q}_{2}\to\mathcal{O}_{l_{0}}\to0.$$ and $\mathcal{Q}|_{\mathcal{C}_{2}}$ is also a sheaf on  $\mathcal{C}_{2}$ with restriction sequence $$0\to\mathcal{O}_{l_{0}}\to\mathcal{Q}|_{\mathcal{C}_{2}}\to\mathcal{O}_{l_{0}}(-1)\oplus T_{3}\to0,\quad or$$

\item[(ii)] $\mathcal{Q}|_{l_{0}}=\mathcal{O}_{l_{0}}(2)\oplus T_{1},$ where $T_{1}$ is a torsion sheaf of length $1,$ $\mathcal{Q}_{2}$ is a sheaf on $\mathcal{C}_{2}$ with restriction sequence $$0\to\mathcal{O}_{l_{0}}(-1)\to\mathcal{Q}_{2}\to\mathcal{O}_{l_{0}}\oplus T_{1}\to0.$$ and $\mathcal{Q}|_{\mathcal{C}_{2}}$ is also a sheaf on $\mathcal{C}_{2}$ with restriction sequence $$0\to\mathcal{O}_{l_{0}}(1)\to\mathcal{Q}|_{\mathcal{C}_{2}}\to\mathcal{O}_{l_{0}}(2)\oplus T_{1}\to0,\quad or$$

\item[(iii)] $\mathcal{Q}|_{l_{0}}=\mathcal{O}_{l_{0}}(2),$ $\mathcal{Q}_{2}$ is a sheaf on $\mathcal{C}_{2}$ with restriction sequence $$0\to\mathcal{O}_{l_{0}}\to\mathcal{Q}_{2}\to\mathcal{O}_{l_{0}}(1)\to0.$$ and $$0\to\mathcal{O}_{l_{0}}(1)\to\mathcal{Q}|_{\mathcal{C}_{2}}\to\mathcal{O}_{l_{0}}(2)\to0,$$

\end{itemize}

\end{theorem}

\begin{proof}
The proof strategy is similar to that of Theorem \ref{cases:c=3}, by arguing on the length $\tilde{z},$ of $\tilde{\mathcal{Z}}$ in \eqref{level:2-2} ; first, note that $0\leq\tilde{z}\leq5.$ 
If $\tilde{z}=0,$ then $\mathcal{Q}_{2}(-2)=\mathcal{O}_{l_{0}}(-3)^{\oplus2}$ and $\chi(\mathcal{Q}|_{l_{0}}(-2))=4.$ Moreover, $\mathcal{Q}|_{l_{0}}(-2)=\mathcal{L}\oplus T.$  By putting $d=deg(\mathcal{L})$ and $t=$length$(T),$ one has $d+t=3.$ On the other hand, $\mathcal{Q}_{3}(-2)=\mathcal{O}_{l_{0}}(-3),$ and from the middle row of \eqref{level:3-2}, $\chi(\mathcal{Q}|_{\mathcal{C}_{2}}(-2))=2.$ Furthermore, $\mathcal{Q}|_{\mathcal{C}_{2}}(-2)$ fits into the restriction sequence
\begin{equation}\label{rest:Q:2} 
0\to\mathcal{L}(-1)\to\mathcal{Q}|_{\mathcal{C}_{2}}(-2)\to\mathcal{L}\oplus T\to0,
\end{equation}
 since $(\mathcal{Q}|_{\mathcal{C}_{2}})|_{l_{0}}=\mathcal{Q}|_{l_{0}}.$ But this implies that $\mathcal{L}=\mathcal{O}_{l_{0}}(-2)$ and length$(T)=5.$ which leads to $\ho^{1}(\mathcal{Q}_{2}(-2))=\mathbb{C}^{4}$ and $\ho^{0}(\mathcal{Q}|_{l_{0}}(-2))=\mathbb{C}^{5},$ contradicting \eqref{cascade:cond} $(c).$ Thus $\tilde{z}\neq0.$ In the same fashion, one proves that $\tilde{z}$ cannot be equal to $4,$ nor $5.$

For the rest of the cases $\tilde{z}=1,2,3$ one can first write $$0\to\mathcal{L}_{2}(-1)\to\mathcal{Q}_{2}(-2)\to\mathcal{L}_{2}\oplus T\to0,$$ from which one has $\mathcal{Q}_{3}=\mathcal{L}_{2}(-1)$ and  $\chi(\mathcal{Q}_{2}(-2))=2d_{2}+t+1$ and set $t=$length$(T)$ and $d_{2}=deg\mathcal{L}_{2}.$ Then, by using the left column of \eqref{level:2-2}, one has the following table: 

\begin{equation}\label{table:cases}
\begin{array}{|c|c|c|c|}
\hline
\tilde{z}&\chi(\mathcal{Q}_{2}(-2))&\bar{z}&\chi(\mathcal{Q}|_{l_{0}}(-2)) \\
\hline
1&-3&4&3 \\
\hline
2&-2&3& 2\\
\hline
3&-1&2& 1\\
\hline
\end{array}
\end{equation}
recall that $\bar{z}$ is the length of $\bar{\mathcal{Z}}$ in \eqref{level:2-2}. In what follows we analise the case in the first row of \eqref{table:cases}. The other cases can be treated similarly. 

When $\tilde{z}=1$ one has $2l_{2}+t+1=-3.$ Since the length $t\geq0,$ one has $d_{2}=-2$ and $t=0$ or $d_{2}=-3$ and $t=2$ or $d_{2}=-4$ and $t=4.$ However, the last two cases cannot hold since $\mathcal{Q}_{3}(-2)=\mathcal{L}_{2}(-1)$ would have degree less than $-3,$ contradicting the first row of \eqref{level:3-2}. Hence we end up with $d_{2}=-2,$ $t=0$ and $\mathcal{Q}_{3}(-2)=\mathcal{O}_{l_{0}}(-3).$ Now, writing $\mathcal{Q}|_{\mathcal{C}_{2}}(-2)$ as in \eqref{rest:Q:2} one should have $2d+t=1.$ Again, by using the fact that $\mathcal{Q}|_{l_{0}}(-2)=\mathcal{L}\oplus T,$ it turns out that the only possibility is $d=-1$ and $t=3.$ Hence 
$$0\to\mathcal{O}_{l_{0}}(-2)\to\mathcal{Q}|_{\mathcal{C}_{2}}(-2)\to\mathcal{O}_{l_{0}}(-1)\oplus T_{3}\to0,$$ and $$0\to\mathcal{O}_{l_{0}}(-3)\to\mathcal{Q}_{2}(-2)\to\mathcal{O}_{l_{0}}(-2)\to0.$$

\end{proof}

\begin{remark} \hspace{10cm}
\begin{itemize}
\item[(I)] When $\mathcal{C}$is primitive, the graded object ${\rm Gr}(\mathcal{Q}),$ associated to the canonical filtration of $\mathcal{Q},$ can be computed from Theorem \ref{cases:c=3}; in case ${\rm (i)}$ one has $${\rm Gr}(\mathcal{Q})=\mathcal{O}_{l_{0}}(2)\oplus\mathcal{O}_{l_{0}}(1)\oplus\mathcal{O}_{l_{0}}$$ hence $\mathcal{Q}$ is a generalised rank $3$ quasi locally free sheaf on $\mathcal{C},$ \cite[Corollary 5.1.4]{Drezet1}. In case ${\rm (ii)},$ we have $${\rm Gr}(\mathcal{Q})=\mathcal{O}_{l_{0}}(1)^{\oplus 2}\oplus\mathcal{O}_{l_{0}}\oplus T_{2},$$ and $T_{2}$ has length $1.$ Hence $\mathcal{Q}$ is, also in this case, a generalised rank $3$ sheaf on $\mathcal{C}.$
\item[(II)] For the non primitive case, one can also compute the graded object, with respect to the chosen filtration, from Theorem \ref{non-primitive}, namely; 
\begin{itemize}
\item[(i)] ${\rm Gr}(\mathcal{Q})=\mathcal{O}_{l_{0}}(-1)\oplus\mathcal{O}_{l_{0}}\oplus\mathcal{O}_{l_{0}}(1)\oplus T_{3},$  and $T_{3}$ has length $3,$ or  %for item ${\rm (i)};$ 

\item[(ii)] ${\rm Gr}(\mathcal{Q})=\mathcal{O}_{l_{0}}(2)\oplus\mathcal{O}_{l_{0}}\oplus\mathcal{O}_{l_{0}}(-1)\oplus T_{1}^{\oplus 2},$ and $T_{1}$ has length $1,$ or  %for item ${\rm (ii)};$ 

\item[(iii)] ${\rm Gr}(\mathcal{Q})=\mathcal{O}_{l_{0}}(2)\oplus\mathcal{O}_{l_{0}}(1)\oplus\mathcal{O}_{l_{0}}$ %for item ${\rm (iii)}.$ 
 \end{itemize}
\item[(III)] Theorems \ref{cases:c=3} and \ref{non-primitive} show that there are at least $7$ components in $\mathcal{M}^{\mathbb{T}}_{\mathbb{P}^{3}}(3);$ $3$ non primitive cases and $4$ primitive ones, counting, both, the horizontal and the vertical Young Diagrams. 
\end{itemize}

\end{remark}

\vspace{0.3cm}

For a given integer $m,$ we now denote by ${\rm p}(m)$ the number of its ($2-$dimensional) partitions. We recall that $l_{\mathcal{Z}}(\nu(c))$ denotes the length of $\mathcal{Z},$ for the multiple structure associated to a partition $\nu(c)$ of $c.$

If we consider the whole set of monomial multiple structures, not only the primitive ones, then we get the following
 
\vspace{0.5cm}

\begin{lemma}\label{minimum-components}
The fixed locus $\mathcal{M}^{\mathbb{T}}_{\mathbb{P}^{3}}(c)$ splits as a union
$$\mathcal{M}^{\mathbb{T}}_{\mathbb{P}^{3}}(c)={\bf \bigcup_{\nu(c)}\mathcal{M}(\nu(c))}.$$
Thus the least number of such irreducible components in $\mathcal{M}^{\mathbb{T}}_{\mathbb{P}^{3}}(c)$ is given by the number ${\rm p}(c),$ of partitions $\nu(c),$ of $c,$ and can be expressed generating function $$\sum_{c=0}^{\infty} {\rm p}(c)x^{c}=\prod_{k=1}^{\infty}\frac{1}{1-x^{k}}.$$
%$${\rm p}(c)\times\sum_{a+b=l_{\mathcal{Z}}}{\rm PL}(a){\rm PL}(b).$$ 
\end{lemma}

\begin{proof}
This is obtained by enumerating the possible Young diagrams, hence enumerating partitions $\nu(c)$ of $c.$ 
\end{proof}
 
We remark that these are completely disconnected components. But as seen, in the case $c=3,$ there might be more than one component for the same partition.

\vspace{0.5cm}

\subsubsection{Stable rank $0$ instanton pair of charge $c$ with primitive support}

We now describe the case in which the support is a primitive multiple line. For the pair $(\mathcal{Q},s)$ of charge $c,$ the associated Cohen-Macaulay curve $\mathcal{C}$ is a primitive multiple curve with ideal whom associated Young digram is a column or a line. As in the last section we treat the case \ytableausetup{mathmode}
\ytableausetup{mathmode}
\begin{ytableau}
\quad & \quad &  \quad & \quad &  \quad & \quad & \none
\end{ytableau}.

\vspace{0.5cm}

This time we have $\mathcal{I}_{\mathcal{C}}=<z_{2},z_{3}^{c}>;$ also, for which $\mathcal{C}$ is a complete intersection. Its resolution is $$0 \to\mathcal{O}_{\p3}(-c-1)\to \mathcal{O}_{\p3}(-c)\oplus \mathcal{O}_{\p3}(-1)\to \mathcal{I}_{\mathcal{C}}\to 0,$$ with Hilbert polynomial $\chi(\mathcal{O}_{\mathcal{C}}(m))=cm-\frac{c(c-3)}{2},$ and length $l_{\mathcal{Z}}=\frac{c(c+1)}{2}.$ 

\vspace{0.5cm}

The canonical filtration of supports is represented by:

\vspace{1cm}

$\hspace{1.4cm}\mathcal{I}_{l_{0}} \hspace{0.7cm}\supset \mathcal{I}_{\mathcal{C}_{2}}\hspace{0.4cm} \cdots \hspace{1cm}\supset\mathcal{I}_{\mathcal{C}_{c-1}} \hspace{1.4cm}\supset \mathcal{I}_{\mathcal{C}}$
\begin{center}
\ytableausetup{mathmode}
\begin{ytableau}
\quad &  \none
\end{ytableau} $\subset$
\ytableausetup{mathmode}
\begin{ytableau}
\quad & \quad
\end{ytableau} $\subset\cdots\subset$ 
\ytableausetup{mathmode}
\begin{ytableau}
\quad & \quad &  \quad & \quad &  \quad 
\end{ytableau} $\subset$
\ytableausetup{mathmode}
\begin{ytableau}
\quad & \quad &  \quad & \quad &  \quad & \quad & \none
\end{ytableau} 
\end{center}
$\hspace{2cm}l_{0}\subset \hspace{0.7cm} \mathcal{C}_{2}\subset\hspace{0.4cm} \cdots \hspace{1cm}\mathcal{C}_{c-1} \subset \hspace{1.8cm} \mathcal{C}$

\vspace{0.5cm}

and we have sequences:

$$0\to\mathcal{I}_{\mathcal{C}_{2}}\to\mathcal{I}_{l_{0}}\to L\cong\mathcal{O}_{l_{0}}(-1)\to0$$
and 
$$ 0\to\mathcal{I}_{\mathcal{C}_{i+1}}\to\mathcal{I}_{\mathcal{C}_{i}}\to L\cong\mathcal{O}_{l_{0}}(-1)^{\otimes i}\to0,$$ hence restrictions sequence $$0\to\mathcal{O}_{l_{0}}(-1)\to\mathcal{O}_{\mathcal{C}}\to \mathcal{O}_{l_{0}}\to0;$$

$$0\to\mathcal{O}_{l_{0}}(-i)\to\mathcal{O}_{\mathcal{C}_{i+1}}\to \mathcal{O}_{\mathcal{C}_{i}}\to0,$$

For $1 \leq i\leq c-1.$

On the other hand, the first canonical filtration of $\mathcal{O}_{\mathcal{C}}$ reads as $$\mathcal{L}_{c+1}=0\subset\mathcal{L}_{c}\subset \cdots\subset\mathcal{L}_{2}\subset\mathcal{O}_{\mathcal{C}},$$
where $\mathcal{O}_{\mathcal{C}}/\mathcal{L}_{i+1}\cong\mathcal{O}_{\mathcal{C}_{i}}.$
\begin{lemma}
The graded sheaf, the generalised degree and the generalised rank of $\mathcal{O}_{\mathcal{C}}$ are given, respectively, by:
$${\rm Gr}(\mathcal{O}_{\mathcal{C}})=\bigoplus_{i=0}^{c-1}\mathcal{O}_{l_{0}}(-i), \hspace{0.5cm} Deg(\mathcal{O}_{\mathcal{C}})=-\frac{c(c-1)}{2} \hspace{0.3cm}\textnormal{ and }\hspace{0.3cm} R(\mathcal{O}_{\mathcal{C}})= c.$$
\end{lemma}

\begin{proof}
By using diagrams 

\begin{displaymath}
\xymatrix@R-1pc@C-1pc{
 & 0 \ar[d]& 0 \ar[d] & &\\
& \mathcal{L}_{i+1}\ar@{=}[r]\ar[d] & \mathcal{L}_{i+1}\ar[d]& & \\
0\ar[r] & \mathcal{L}_{i}\ar[r]\ar[d] & \mathcal{O}_{\mathcal{C}}
 \ar[d]\ar[r]& \mathcal{O}_{\mathcal{C}_{i-1}} \ar[r]\ar@{=}[d]&0 \\
0\ar[r]& \mathcal{O}_{l_{0}}(-i+1)\ar[d]\ar[r]
& \mathcal{O}_{\mathcal{C}_{i}}\ar[r] \ar[d]&\mathcal{O}_{\mathcal{C}_{i-1}}\ar[r]&0 \\
&0&0&&
}
\end{displaymath}
 one gets the graded sheaf. The generalised degree and rank follow easily by applying their definitions.
\end{proof}

By Theorem \ref{Euler} one gets $\chi(\mathcal{O}_{\mathcal{C}}(m))=cm-\frac{c(c-3)}{2}.$ Thus the graded sheaf associated to $\mathcal{O}_{\mathcal{C}}$  is ${\rm Gr}(\mathcal{O}_{\mathcal{C}})=\bigoplus_{i=0}^{c-1}\mathcal{O}_{l_{0}}(-i),$  and the generalised rank and degree are, respectively, $R(\mathcal{O}_{\mathcal{C}})=c,$ $Deg(\mathcal{O}_{\mathcal{C}})=-\frac{c(c-1)}{2}.$

\vspace{0.4cm}

We remark that we do not know whether the above fixed components intersect the closure of the framed locally free instanton moduli, in general. We think that this problem is related to {\em reachability} of sheaves, on multiple structure \cite{Drezet3}. Nevertheless, for charge $c=1,$ the answer is positive; the sheaf $\ker(\mathcal{O}_{\p3}^{\oplus 2}\to\mathcal{Q}),$ in \ref{charge1}, is in the closure of the moduli of locally free framed instanton bundles \cite[\S 6]{JMT2}. Furthermore, if $c=2$ one can deform the (monomial) double curve into a union of two curves intersecting at a point. Moreover, the moduli space of instantons of charge $c=2$ is irreducible as proved in \cite[Propositioln. 7]{JMT}. Thus the $0$ instanton sheaf should deform, from sheaf on the double curve, to a sheaf on the reduced curve. Hence, the fixed component is in the closure of the moduli space of framed instantons. For higher values of the charge, this is a difficult problem to answer. Since we think this is true, we close this notes by writing

\begin{conjecture}
The fixed components, under the lifted toric action on $\p3,$ intersect the closure of the locally free component in the moduli space of framed instantons.
\end{conjecture}

\vspace{0.3cm}

\paragraph{{\bf Acknowledments}}

I would like to thank Marcos Jardim for the useful discussions we had during my few shorts visits to IMECC-UNICAMP and his valuable remarks about the first draft. I am also thankful to the referees of the \emph{Pacific Journal of Mathematics}, for their corrections and suggestions.

\vspace{0.5cm}

%---------------------------------------------------------------

%\vspace{3cm}

\vspace{0.2cm}

{\small

Abdelmoubine Amar Henni   

Departamento de Matem\'atica MTM - UFSC

Campus Universit\'ario Trindade CEP 88.040-900 Florian\'opolis-SC, Brazil

e-mail: henni.amar@ufsc.br   

}

\end{document}